\documentclass[12 pt]{amsart}

\usepackage[margin=2.3cm]{geometry} 
\usepackage{graphicx}
\usepackage{amssymb,latexsym}
\usepackage{amsmath,amsfonts,amstext,bbm}
\usepackage[utf8]{inputenc}
\usepackage{xcolor}

\newtheorem{theorem}{Theorem}
\newtheorem{lemma}{Lemma}
\newtheorem{proposition}{Proposition}

\DeclareMathOperator{\E}{\mathbb{E}}
\DeclareMathOperator{\Pp}{\mathbb{P}}
\DeclareMathOperator{\R}{\mathbb{R}}
\DeclareMathOperator{\N}{\mathbb{N}}
\DeclareMathOperator{\Z}{\mathbb{Z}}
\DeclareMathOperator{\T}{\mathbb{T}}
\DeclareMathOperator{\V}{\mathbb{V}}
\DeclareMathOperator{\Oh}{\mathcal{O}}

\begin{document}
	\title[RANDOM EXPONENTIAL SUMS AND LATTICE POINTS IN REGIONS]
	{Random exponential sums and lattice points in regions}
	
	\author{FARUK TEMUR, CİHAN  SAHİLLİOĞULLARI}
	\address{Department of Mathematics\\
		İzmir Institute of Technology\\ Urla \\İzmir \\ 35430\\Turkey  }
	\email{faruktemur@iyte.edu.tr, cihansahilliogullari@iyte.edu.tr}

	\date{November  10, 2024}

	\begin{abstract}
		In this article we  study two fundamental problems on exponential sums via   randomization of frequencies  with stochastic processes. These are the Hardy-Littlewood majorant problem, and  $L^{2n}(\T), \ n\in \N$ norms of exponential sums, which can also be interpreted as solutions of diophantine equations or lattice points on surfaces. We  establish  connections to the well known problems on lattice points in regions such as the Dirichlet divisor problem. 
	\end{abstract}
	
	\maketitle

	\section{introduction}

		In this work\footnote{The results of this article will constitute part of the second authors PhD thesis \cite{cs}. For deeper and wider background on  these topics, more heuristics, and more detailed proofs we refer the reader to this thesis.} we will study random exponential sums of the form
	\begin{equation}\label{int0}
		\Big\|\sum_{j\in A} a_je^{2\pi i yX_j}\Big\|_{L^{p}(\T)}^{p} 
	\end{equation}	
	where $A$ is a finite subset of integers, and $X$ is a stochastic process on a probability space $(\Omega,\mathcal{F},\Pp).$   We will use these sums to investigate two esteemed problems of harmonic analysis and analytic number theory, the Hardy-Littlewood majorant problem, and  $L^{2n}(\T), \ n\in \N$ norms of exponential sums.  We will relate these sums for $p=2n$ and $a_j=0,1$ to problems on lattice points in regions such as the Dirichlet divisor problem.  The regions that we will encounter  have the form of shells around hypersurfaces. This can be seen as an analogue of the interpretation of $L^{2n}(\T), \ n\in \N$ norms of deterministic exponential sums as counting lattice points on surfaces.  But in our case establishing this interpretation is rather involved.

	The study of random exponential sums is a century old topic that is still immensely relevant as evidenced by the   vast number of citations to  \cite{k} which is a standard reference in this area.  Since the pioneering works of Paley and Zygmund \cite{pz1,pz2}, exponential sums    have often been   randomized by multiplying the coefficients  with independent identically distributed random variables $X_j$:
	\begin{equation}\label{}
		\Big\|\sum_{j\in A} a_jX_je^{2\pi i jy}\Big\|_{L^{p}(\T)}^{p}. 
	\end{equation}
	And  often the distribution of $X_j$ is chosen to be either Bernoulli with values  $\pm 1$, or  uniform  on $\T$,  or Gaussian. To study  our problems however we must randomize the frequency set $A$, which can be done either by multiplying the coefficients with selector variables, that is Bernoulli variables that take values $0$ and $1$, or by directly randomizing the frequency as we do in  this work.  Selector variables were employed by Mockenhaupt and Schlag \cite{ms} to study the Hardy-Littlewood majorant problem. An advantage  of our approach is that we can do  randomization around a fixed set without much destroying  the structure of it. For instance the sets $\{ N(j^d)\}_{j\in \N}$ for $d\geq 3$ will still look much like the set $\{j^d\}_{j\in \N}$, whereas had we randomized with selector variables they would  have lost all similarity to it.

	Precedents for randomization of frequency by stochastic processes exist in abundance. They emerge naturally in the study of  Gaussian Fourier series, and of Brownian images. For example in Chapter  14 of \cite{k}  for a process  $F(t)$ obtained by Gaussian Fourier series, and $\theta$ some measure on a subset $E\subseteq\T$ the expression
	\begin{equation}\label{int-2}
		\E\Big\|\int_E e^{2\pi i uF(t)}d\theta(t)\Big\|_{L^{2}(\R^d)}^{2} 
	\end{equation}
	is calculated to study the measure of $F(E).$ Also in Chapter 17  of the same work, to compute the Fourier  dimension of the  Brownian image $W(E)$ of a set $E$,   decay estimates as $u\rightarrow \infty$ are obtained for 
	\begin{equation}\label{int-3}
		\Big|\int_0^\infty e^{2\pi i uW(t)}d\theta(t)\Big|,
	\end{equation}
	where $\theta$ is a positive measure on the positive half line.  Estimates similar to \eqref{int-2},\eqref{int-3} went on to be employed by researchers studying Gaussian Fourier series and Brownian images ever since, see \cite{eh,fmd,fs, gg,g,mu,x} and references therein. Note that  in \eqref{int-2} and especially in  \eqref{int-3} by choosing $\theta$ to be a sum of Dirac delta measures we can obtain a random exponential sum similar to those we consider in this work. Though we  should also remark that owing to the purpose for which these estimates are utilized, $\theta$ in its original contexts is generally considered to be a measure supported on a larger set, so as to enable  decay, whereas sums in this article are always periodic. Hence randomization of frequency by stochastic processes not only has certain advantages peculiar to itself, but is also very natural with a long and established history.

	 Frequency randomization by stochastic processes  brings  two  difficulties that do not exist in  coefficient randomization. The first is  that some frequencies will repeat, therefore the Fourier series structure is violated and certain basic tools like the Plancherel theorem cannot be directly applied.   The second is, the random variables $X_j$ obtained from a process will not be independent for many important classes of processes, the most important for us being  the Poisson process.   We however surmount these difficulties by concentrating on differences of the form $X_{j_2}-X_{j_1},  \ X_{k_2}-X_{k_1} $ which are independent for any independent increment process, such as  Poisson  and Wiener processes, whenever $(j_1,j_2)\cap (k_1,k_2)$=$\emptyset$.  We delicately decompose our sums so as to obtain such independent differences whenever possible. When this is not possible, we use combinatorial and arithmetic counting arguments   to count the number of intersections, and if possible  reduce them to the non-intersecting case.

	$L^{2n}, n\in \N$ norms of exponential sums correspond to solutions of diophantine equations: letting $A\in \Z$ be a finite set, and ${\bf j,k}\in A^n$ with ${\bf j}=(j_1,j_2,\ldots,j_n), \ {\bf k}=(k_1,k_2,\ldots,k_n)$

	\begin{align}\label{int1}
		\Big\|\sum_{j\in A} e^{2\pi i yj}\Big\|_{L^{2n}(\T)}^{2n} &= \big| \big\{({\bf j,k})\in A^{2n}:\  \  \sum_{i=1}^n j_i=  \sum_{i=1}^nk_i    \ \  \big\}\big| \\&= \sum_{m\in \Z} \big| \big\{ {\bf j}\in A^{n}:\   \sum_{i=1}^n j_i=m      \big\}\big|^2 \label{int2}
	\end{align}
	where $|\cdot|$ denotes cardinality. Estimates on the size of the sets in \eqref{int1},\eqref{int2} in terms of  $|A|$ have always been of great interest to mathematicians. We will briefly summarize the essentials we need of this  broad subject, for much more detail see \cite{bhb,hb,hb2,mar,s,sw,v} and references therein.   Sets in \eqref{int1} always include $\sim  n! |A|^n$ trivial solutions coming from choosing ${\bf k}$ a permutation of ${\bf j}$. So the problem of estimating their size reduces   to estimating the cardinality of nontrivial solutions. For certain $A,n$ these nontrivial solutions may not even exist. The sets in \eqref{int2} are representations of $m$ by elements of $A^n$, and we will denote their cardinality by  $R_{A^n}(m).$ Certain of these sets may be empty, but  by permutations, if a set has one element ${\bf j}$ all permutations of it are also inside.

	Most often $A$ is chosen to be  $d$-powers of the set $\{1,2,3\ldots,M \}$, and  an estimation in terms of $M$ is demanded. In this case,  the set of \eqref{int1} turns into: 
	\begin{equation}\label{int3}
		\big| \big\{({\bf j,k}):  \    1\leq j_i,k_i\leq M,  \quad   j_1^d+j_2^d+\cdots +j_n^d=	k_1^d+k_2^d+\cdots +k_n^d  \big\}\big|, 
	\end{equation}
	The sets in \eqref{int2}, the cardinalities of which in this particular case we will denote by $R_{n,d,M}(m)$, are representations  of $m$ by $n$ $d$-powers:
	\begin{equation}\label{int4}
		R_{n,d,M}(m):=	\big| \big\{{\bf j}:  \    1\leq j_i\leq M, \quad  j_1^d+j_2^d+\cdots +j_n^d=	m  \big\}\big|. 
	\end{equation}
	Plainly these problems can also be viewed as finding lattice points on a variety in affine or projective space.
	The cases of small $n$, such as  $n=2,3,4$ have been studied even more intensely. 
	
	To have an opinion on the size of the sets in  \eqref{int3},\eqref{int4} there is the following  well known heuristic argument.   $A^n$ has $M^n$  elements, but any sum $	j_1^d+j_2^d+\cdots +j_n^d$ is an integer between $1$ and $nM^d.$ Therefore, on average there are $n^{-1}M^{n-d}$ representations any integer $1\leq m \leq nM^d $. So, for $d\geq n$ it is reasonable to expect that   $R_{n,d,M}(m)\lesssim 1$  for any $m$, and  that trivial solutions dominate in \eqref{int3}, which is referred to as paucity of nontrivial or nondiagonal solutions.  This heuristics is behind the Hardy-Littlewood Hypothesis $K$, which is essentially the claim 
	 $R_{n,n,M}(K)=\mathcal{O}(K^{\varepsilon})$. This conjecture is well known to be true for $n=2$, but was disproved by Mahler \cite{m}  for $n=3.$ Yet it may still be true for various other values of $n.$  There is also the conjecture, relying on the same heuristics, that the cardinality of the set in \eqref{int3} is bounded by  $C_{\varepsilon}M^{\varepsilon}\max\{M^n,M^{2n-d}\},$
	see \cite{v}. For $n=2$ this is known, but even for $n=3$ it is open, and a topic of great interest. Indeed great effort \cite{bhb,hb,s} has been  directed to just obtaining bounds $M^{7/2-\delta}$ for $d$  high enough. The bound $M^{3+\varepsilon}$ has been achieved for $d>25$ by Salberger in \cite{s}, and no reduction in $d$ has been possible in the intervening two decades. For higher values of $n$ even less is known, see \cite{mar,sw}. In \cite{sw} Salberger and Wooley prove the only result in the literature  for generic $n$, and establish the paucity of nontrivial solutions, and hence the conjecture, for $d\geq (2n)^{4n}.$ 
	
	If we take $d$-powers of an arbitrary finite set $A$, these problems become much harder, and even  the $n=2$ case is wide open. In this case Chang \cite{c} conjectures that 
	\begin{equation}\label{chang}
		\Big\|\sum_{j\in A} e^{2\pi i yj^2}\Big\|_{L^{4}(\T)}^{4}\leq C_{\varepsilon}|A|^{2+\varepsilon}.
	\end{equation} 
This conjecture is  closely related to many other problems of harmonic analysis, additive combinatorics and number theory, see \cite{c,cg, san} for a broad view of these connections.
  The bound $|A|^3$ is trivial for \eqref{chang}, and there has been no power type  improvement on this, see \cite{c,san} for the improvements that have been possible. Randomized versions of this conjecture will be some of our main objects of study in this work, and we will confirm Chang's conjecture on average, and uncover similar behaviour for any $d,n.$

	Another old and esteemed problem on exponential sums is the Hardy-Littlewood majorant problem which  concerns precisely non-even $L^p$ norms.  Let $A\subseteq \{1,2,\ldots,N\}$. Then for $|a_j|\leq 1,  j\in A  $ and $p=2n$ we have 
	\begin{equation*}\label{int5}
		\Big\|\sum_{j\in A} a_je^{2\pi i yj}\Big\|_{L^{p}(\T)}\leq 	\Big\|\sum_{j\in A} e^{2\pi i yj}\Big\|_{L^{p}(\T)}.
	\end{equation*} 
	Hardy and Littlewood asked what happens  for arbitrary $p>2.$  To make things clearer we define the constants
	\begin{equation*}
		\sup_{|a_j|\leq 1 }  \Big\|\sum_{j\in A} a_je^{2\pi i yj}\Big\|_{L^{p}(\T)}= B_p(A)	\Big\|\sum_{j\in A} e^{2\pi i yj}\Big\|_{L^{p}(\T)},
	\end{equation*} 
	and 
	\begin{equation*}
		B_p(M):=\sup_{A\subseteq \{1,2,\ldots,M\}}B_p(A).	
	\end{equation*}
	We note that $B_p(M)$ is an increasing quantity. Already Hardy and Littlewood knew that $B_3(A)>1$ for certain sets $A$.  So it is not reasonable to expect in general what is true for even exponents. But it may still be expected to have $B_p(M)$  bounded  by a constant independent of $M$. Bachelis \cite{bac} disproved this, and showed that for any $p$ that is not an even integer $B_p(M)\rightarrow \infty$ as $M\rightarrow \infty$. After this for a time it was hoped that $B_p(M)\leq C_{\varepsilon}M^{\varepsilon}$ may be possible. But the works of  Green-Rusza \cite{gr} and Mockenhaupt-Schlag \cite{ms} disproved even this. Thus  there is no hope for results that will hold for all sets, but for specific subsets of $\N$ positive results are still possible, and important. Indeed it is known  by the work of Mockenhaupt  \cite{moc} that   for certain sets  $\Gamma\subseteq \{1,2,\ldots,M\}$ obtained by discretizing and projecting the sphere $S^{n-1}$ the bound 
	$B_p(\Gamma)\leq C_{\varepsilon}M^{\varepsilon}$ would imply the famous restriction conjecture in harmonic analysis. For the restriction conjecture, its connections and variants see  \cite{dem,ft1,ft2,ft3}. So studying the Hardy-Littlewood majorant problem for specific subsets is  very important and continues to this day, see for example \cite{d,gre,kmt,kre}. This is also very meaningful, for  Mockenhaupt-Schlag in an article \cite{ms}  that  served as the starting point of our investigations proved that  generic sets  satisfy the Hardy-Littlewood majorant property. In this they built upon arguments of Bourgain \cite{bou1}. More  precisely let $\xi_j$ be i.i.d selector variables with distribution
	$\Pp[\xi_j=1]=M^{-\delta}$ and $\Pp[\xi_j=0]=1-M^{-\delta}$ where $0<\delta<1.$  Then for any $\varepsilon>0$, as $M\rightarrow \infty$ 
	\begin{equation}\label{intms}
		\Pp\Big[  \sup_{|a_j|\leq 1}\Big\| \sum_{j=1}^{M} a_j\xi_je^{2\pi i yj}   \Big\|_{L^p(\T)} \geq M^{\varepsilon}  \Big\| \sum_{j=1}^{M} \xi_je^{2\pi i yj}   \Big\|_{L^p(\T)} \Big]\rightarrow 0.
	\end{equation} 
	
	One area of study  that is very much connected with the Hardy-Littlewood majorant problem is the $\Lambda(p)$ sets.  A set $A\subseteq \Z $ is a $\Lambda(p),\  p>2 $ set, if for any exponential sum on it
	\begin{equation*}
		\Big\|\sum_{j\in A} a_je^{2\pi i yj}\Big\|_{L^{p}(\T)} \leq C_{A,p}	\Big\|\sum_{j\in A} a_je^{2\pi i yj}\Big\|_{L^{2}(\T)} =  C_{A,p}\Big[	\sum_{j\in A} |a_j|^2\Big]^{1/2}.
	\end{equation*} 
	Plainly this is true for any finite set, so this concept is nontrivial for infinite sets $A$ only. Now suppose $A$ is a $\Lambda(p)$ set. Then for coefficients $a_j$ of size at most $1$,
	\begin{align*}
		\Big\|\sum_{j\in A} a_je^{2\pi i yj}\Big\|_{L^{p}(\T)} \leq C_{A,p}	\Big\|\sum_{j\in A} a_je^{2\pi i yj}\Big\|_{L^{2}(\T)} 
		\leq  C_{A,p}	\Big\|\sum_{j\in A} e^{2\pi i yj}\Big\|_{L^{2}(\T)}\leq C_{A,p}	\Big\|\sum_{j\in A} e^{2\pi i yj}\Big\|_{L^{p}(\T)},
	\end{align*}
	showing that this set has the Hardy-Littlewood majorant property with $B_p(A)\leq C_{A,p}$. Bourgain \cite{bou1} proved that a generic subset of   $\{1,2,\ldots,M\}$ of size $[M^{\delta}], \   0<\delta \leq 2/p$ is a $\Lambda(p)$ set, and hence satisfies the Hardy-Littlewood majorant property. In \cite{bou2} he extended this result to generic subsets of powers and primes.  Genericity meant in these  works  is   similar to \eqref{intms}. See also the very recent work \cite{djr} extending and furthering this line of research.
	Lastly, we remind the important conjecture of Rudin \cite{ru} that squares is a $\Lambda(p)$ set for $p<4$. This conjecture, if true, would imply Chang's conjecture.

	Below we  introduce our results, and the
	   following  notation  will help us articulate them. The notation $a\lesssim b$ means there exists a positive constant $C$ such that $a\leq Cb$, and $a\gtrsim b$ is defined analogously. We denote $a\approx b$ if both of these are true. For two functions $f,g$ the notation $f\sim g$ means $f(x)/g(x)\rightarrow 1 $ as $x\rightarrow \infty.$ We will also employ the standard asymptotic  notation $ \Oh, o,\Omega $. We use $\E$ to denote the expectation and $\V$ for variance. We let $\Z_+$ denote the nonnegative integers, and    $\R_+$ the nonnegative real numbers.

	\begin{theorem}
		Let $\{X_j\}_{j\in \mathbb{N}}$ be  a stationary process taking only integer values. 	Let the probability mass function of the  random variables in our process be  denoted by $\mu:\Z\rightarrow \R.$  Let $A\subset \N$  be any finite nonempty subset.  Then for any $1\leq p<\infty$ and any $\varepsilon>0$ we have
		
		\begin{equation}\label{s1}
 |A|^p	\lesssim_{\mu,p}	\mathbb{E}\Big\|\sum_{j\in A} e^{2\pi i yX_j}\Big\|_{p}^p \leq |A|^p,
		\end{equation}
		
		\begin{equation}\label{s2}
			\mathbb{E}\sup_{|a_j|\leq 1}\Big\|  \sum_{j\in A} a_je^{2\pi i yX_j}  \Big\|_{p}^p \lesssim_{\mu,p} \mathbb{E}\Big\|  \sum_{j\in A} e^{2\pi i yX_j}   \Big\|_{p}^p,
		\end{equation}
		
		\begin{equation}\label{s3}
			\lim_{|A|\rightarrow \infty}\mathbb{P}\Big[\sup_{|a_j|\leq 1} \Big\|  \sum_{j\in A} a_je^{2\pi i yX_j}  \Big\|_{p} \geq |A|^{\varepsilon}\Big\|  \sum_{j\in A} e^{2\pi i yX_j}   \Big\|_{p} \Big]=0.
		\end{equation}
		Analogues of these results for $p=\infty$ also hold.
	\end{theorem}
	Heuristically a  process being stationary means it repeats the same values with the same probabilities. Thus the exponential sum in \eqref{s1} may be expected to behave like an $|A|$-fold  sum  of the same exponential $e^{2\pi i yX_j}.$ Our result \eqref{s1} confirms this heuristic, and it is strong enough to yield \eqref{s2},\eqref{s3} immediately.  As seen in there \eqref{s1} is about the  $L^p$ norms of random exponential sums, \eqref{s2} is the  Hardy-Littlewood majorant property on average, while \eqref{s3} is the Hardy-Littlewood majorant property for generic sets as in the contexts of  Bourgain and Schlag-Mockenhaupt. This theorem is obtained using only basic techniques of harmonic analysis and probability, and in its proof we illustrate these techniques.
	
	Our second result gives similar estimates for the case of the Poisson process with the same techniques. As its proof is relatively similar, it will only be summarized, with differences especially indicated. Recall that for $1<p<\infty$
	\begin{equation}\label{int6}
		\Big\|\sum_{j=1}^M e^{2\pi i yj}\Big\|_{p}^p\approx_p M^{p-1}.
	\end{equation}
	Let us see our second  theorem. 
	\begin{theorem}
		Let $\{N(t)\}_{t\geq 0}$ be  a Poisson process of intensity $1$.  Let $\{a_j\}_{j\in \mathbb{N}}$ be any complex sequence with $|a_j|\leq 1, \ j\in \mathbb{N} $. Then for any $2\leq p< \infty$ we have
		
		\begin{equation}\label{p1}
			\mathbb{E}\Big\|\sum_{j=1}^M e^{2\pi i yN(j)}\Big\|_{p}^p \approx_p M^{p-1}.
		\end{equation}
	 	Analogues of \eqref{s2},\eqref{s3} hold as well. Analogues of these results for $p=\infty$ also hold.

	\end{theorem}
	Heuristics behind this is that $N(j)\sim j$ with high probility: this follows from the well known facts  $\E[N(j)]=\V[N(j)]=j$. So \eqref{p1}  confirms this heuristics, and gives the same estimate as in  \eqref{int6}. Again this estimate is strong enough to yield the other results of the theorem.

	We can prove a similar theorem for the simple random walk, by essentially the same arguments. One difficulty is that the simple random walk is not increasing, which can be overcome by the Doob martingale inequality. One missing point of both Theorem 2 and Theorem 3 compared to Theorem 1 is that  Theorem 1 is proved for an arbitrary subset $A\subseteq \{1,2,\ldots, M\}$ while the other two theorems are only valid for the whole set $\{1,2,\ldots,M\}$.  It would be very nice to prove these theorems for arbitrary subsets as well, but this seems to be rather difficult. Another point is that we cede the range $1\leq p<2.$
	\begin{theorem}
		Let $\{R(j)\}_{j\in\Z_+}$ be  a simple random walk.   Then for any $2\leq p< \infty$ we have
		
		\begin{equation}\label{r1}
			\mathbb{E}\Big\|\sum_{j=1}^M e^{2\pi i yR(j)}\Big\|_{p}^p \approx_p M^{p-1/2}.
		\end{equation}
			Analogues of  \eqref{s2},\eqref{s3} hold  as well.	Analogues of these results for $p=\infty$ also hold.
		
	\end{theorem}
	From the facts $\E[|R(j)|] \approx \sqrt{j}$, and $\V[R(j)]=j$ we may heuristically say  $|R(j)|\approx\sqrt{j}$  with high   probability, and from this guess \eqref{r1}. We can summarize \eqref{s1},\eqref{p1},\eqref{r1}  heuristically as follows.   If for a process  $X(j)$ we have $X(j)\approx j^{1-\alpha}, \ 0\leq \alpha \leq 1$ with high probability, then we may see this as  a value  being repeated $M^{\alpha}$  times,  and write
	\begin{equation}\label{h1}
		\mathbb{E}\Big\|\sum_{j=1}^M e^{2\pi i yX(j)}\Big\|_{p}^p \approx 	\mathbb{E}\Big\|M^{\alpha}\sum_{j=1}^{M^{1-\alpha}} e^{2\pi i yj}\Big\|_{p}^p \approx  M^{p\alpha+(p-1)(1-\alpha)}= M^{p-1+\alpha}.
	\end{equation}
	This  is a very useful heuristic to guess what happens for other processes.

	We then  move on to more difficult problems. In these we take the stochastic process to be  Poisson, and investigate subsets of powers, subsets of what we may call  equally spaced powers, and the sets introduced in Green-Ruzsa work \cite{gr} that we will call the Green-Ruzsa sets. We  start with $L^4$ estimates for powers, as these illustrate our ideas in a simpler context.
	
	\begin{theorem}
		Let $\{N(t)\}_{t\geq 0}$ be  a Poisson process of intensity $1$. Let $d\geq 2$ be an integer. Let $A\subset \N$ be a finite nonempty subset.  Then for any $2\leq p\leq 4$ we have
		
		\begin{equation}\label{pp1}
			|A|^{p/2} \leq	\mathbb{E}\Big\|\sum_{j\in A} e^{2\pi i yN(j^d)}\Big\|_{p}^p \lesssim_p 	\begin{cases}|A|^{p/2}\log^{3p/8}(1+|A|)	\ &\text{for} \ d=2  \\  |A|^{p/2} \ &\text{for}  \ d\geq 3. \end{cases}
		\end{equation}
		Analogues of  \eqref{s2},\eqref{s3}  hold for $2\leq p\leq 4$.
		
	\end{theorem}

	Note that the theorem is proved for an arbitrary $A\subseteq \{1,2,\ldots,M\},$ and therefore proves that Mei-Chu Chang's conjecture is correct in an average sense. The reason why this very difficult conjecture becomes reachable with randomization is that  randomizing is much like averaging, that is, rather than considering just one $j^d$ we consider $N(j^d)$ which heuristically takes  every integer value  in $(j^d-j^{d/2},j^d+j^{d/2})$ each with $j^{-d/2}$ probability.  This gives  a gain of  $j^{-d/2}$ in the  sum we want to evaluate. Combined with a good counting argument  for lattice points in shells this suffices to obtain our result.
	In this theorem proving \eqref{pp1} for  $p=4$ is the key step, all other claims of the theorem then follow. As carrying out this proof we will demonstrate  delicate decompositons, combinatorial arguments, counting arguments and finally the connection to lattice point counting in shell type regions that we will develop further to prove subsequent theorems.

	The next theorem is for $L^4$ norms of equally spaced powers, which we study mainly to compare and contrast   with Theorem 4.

	\begin{theorem}
		Let $\{N(t)\}_{t\geq 0}$ be  a Poisson process of intensity $1$.	Let $r>0$ be a real number. Then
		
		\begin{equation*}\label{lss}	
			\mathbb{E}\Big\|\sum_{j=1}^M e^{2\pi i yN(jM^r)}\Big\|_{2}^2 \approx M,	 \qquad	 M^{2}+M^{3-r} \lesssim \mathbb{E} \Big\|\sum_{j=1}^M e^{2\pi i yN(jM^r)}\Big\|_{4}^4 \lesssim_r M^{2}\log M+M^{3-r}.
		\end{equation*}
		
	\end{theorem}

For $r<1$, plainly $M^{3-r}$ dominates, and the logarithmic loss is unimportant. For $r>1$ it should be possible with sharper methods we used for other theorems to remove the logarithmic loss. For $r=1$ it might not be possible to remove it, for it is felt even in heuristic computations.	
One  issue  that this theorem highlights is  how randomization  eliminates arithmetic progression structure so that instead of having  $M^3$  for the $L^4$ estimate we have $M^{3-r}$. Also proving this theorem is somewhat  easier than proving \eqref{pp1}, for the arithmetic of dealing with these equally spaced powers is  easier.

	  Green and Ruzsa  \cite{gr} introduced the following sets for which the Hardy-Littlewood majorant property fails  polynomially. 
	For any integer $D\geq 5$ and $ k\in \N $ the Green-Ruzsa set  is 
	\begin{equation*}
		\Lambda_{D,k}:= \Big\{\sum_{j=0}^{k-1}d_j D^j \  \Big|    \   d_j\in \{0,1,3\}  \Big\}.
	\end{equation*}
	It would be interesting to see what happens if we randomize these sets via the  Poisson process, and then look at the Hardy-Littlewood property. Our next theorem studies this and provides the answer that Hardy-Littlewood property holds almost surely.  Indeed this follows merely from the sparsity of these sets, so our theorem is stated for a very large class of sparse sets.
	
	\begin{theorem}
		Let $\{N(t)\}_{t\geq 0}$ be  a Poisson process of intensity $1$. Let $A\subset \Z_+ $ be a finite nonempty subset such that  for any $M\in \N,  n\in \Z$ 
		\begin{equation*}
			\big|A \cap [n-M,n+M]\big| \leq C_AM^{\alpha},
		\end{equation*}
	with $C_A$ a constant that depends only on $A$. 	If $\alpha <1/3$, then  for any $2\leq p\leq 4$, and any $\varepsilon>0$ we have
		
		\begin{equation}\label{gr1}
		|A|^{p/2}\leq	\mathbb{E}\Big\|\sum_{j\in A} e^{2\pi i yN(j)}\Big\|_{p}^p \lesssim_{\alpha,C_A} |A|^{p/2},
		\end{equation}
			\begin{equation}\label{gr2}
			\mathbb{E}\sup_{|a_j|\leq 1} \Big\|  \sum_{j\in A} a_je^{2\pi i yN(j)}  \Big\|_{p}^p \lesssim_{\alpha,C_A} \E \Big\|  \sum_{j\in A} e^{2\pi i yN(j)}   \Big\|_{p}^p,
		\end{equation}
	\begin{equation}\label{gr3}
			\lim_{|A|\rightarrow \infty}\mathbb{P}\Big[\sup_{|a_j|\leq 1} \Big\|  \sum_{j\in A} a_je^{2\pi i yN(j)}  \Big\|_{p}^p \geq |A|^{\varepsilon}\Big\|  \sum_{j\in A} e^{2\pi i yN(j)}   \Big\|_{p}^p \Big]=0.
		\end{equation}
	where the limit is taken over $A$ for which $\sup_A C_A<\infty$ and $\sup_A \alpha<1/3$. 
	\end{theorem}
	So randomization destroys the specific, fractal-like  structure of the Green-Ruzsa sets  that violates the Hardy-Littlewood property.
	We will show in the pertinent section  that the Green-Ruzsa sets satisfy the sparsity property demanded by the theorem.   In the proof instead of using the arithmetic structure of the set as in the proof of  Theorem 4 we use its sparsity.

	We then arrive at the generic case $p=2n$. In order to prove this case  we make our arguments abstract and streamlined.

	\begin{theorem}
		Let $\{N(t)\}_{t\geq 0}$ be  a Poisson process of intensity $1$. Let $d\geq n\in \N$ be an integer. Let $A\subset \N$ be a finite nonempty subset.  Then  we have
		
		\begin{equation}\label{p2n1}
		\begin{aligned}	
			\mathbb{E}\Big\|\sum_{j\in A} e^{2\pi i yN(j^d)}\Big\|_{2n}^{2n}	\lesssim_{d,n} \begin{cases}\max\big\{|A|^n, |A|^{2n-\frac{d}{2}-1}\log^{1+\frac{1}{d}}(1+ |A|) \big\}   &\text{if} \quad  d\equiv 2  \pmod 4, \\
			\max\big\{|A|^n,	|A|^{2n-\frac{d}{2}-1}\log^{\frac{1}{d}}(1+ |A|)\big\}     &\text{else}. 
			\end{cases}
		\end{aligned}
	\end{equation}
		\end{theorem}
	In comparison, the deterministic results for this general case of Salberger and Wooley  \cite{sw} obtain the bound $|A|^{n+\varepsilon}$ for $d\geq (2n)^{4n}$, while in our randomized setting we  obtain it as soon as $d\geq 2n-2$. For the specific $n=3$ case we obtain the  $|A|^{3+\varepsilon}$ bound as soon as $d>3$, while in Salberger  \cite{s} this is $d>25.$ In \cite{s}    improvement over  even the $7/2$ exponent  takes place only for $d>8.$ Also our set $A$ is arbitrary, and not just natural numbers up to $M$.

	We now would like to mention the connection that is made to well-known problems on lattice points in regions. 	After long and arduous decompositions and dealing with auxiliary sums, the main sums of Theorems 4,7 reduce to  a lattice point problem of the following form
	\begin{equation}\label{ss2}
		\sup_{D\leq E\leq D^2}	\big|\big\{ (j,k) \in \N^2: j<k,  \   |k^d-j^d-E| <D \big\}\big|
	\end{equation}
	for large $E,D$.
	This estimate can be obtained either directly, or by viewing this set  as  a margin of the set   
	\begin{equation}\label{os}
		\big\{ (j,k) \in \N^2:   \   0<k^d-j^d <D \big\}.
	\end{equation}
	But estimates on sets of this type do exist in number theory literature. Indeed  they are some of the oldest and most famous problems in number theory. We give a brief background on these, and then state the result to be used to deduce an estimate on \eqref{ss2}.

	We start with the Dirichlet divisor problem. For $x\geq 1$ the divisor summatory function $D(x)$ is defined by
	\begin{equation*}
		D(x):=\sum_{n\leq x}d(n)=|\{ (a,b)\in \N^2:  ab\leq x \}|,
	\end{equation*} 
	where $d(n)$ is the number of positive divisors of $n.$ Plainly this is a lattice point problem.  Dirichlet \cite{dir} proved that 
	\begin{equation*}
		D(x)=x\log x +(2\gamma-1)x+\Delta(x),   \qquad    \Delta(x)=  \mathcal{O}(\sqrt{x}),
	\end{equation*} 
	where $\gamma=0.577\ldots$ is the Euler-Mascheroni constant. 
	The Dirichlet divisor problem is then to study the error term $\Delta(x).$
	Tremendous effort has been poured into this  problem  for two centuries. Dirichlet's result can be obtained by carefully rearranging the  lattice point sum.  Voronoi \cite{vor} obtained $\mathcal{O}(x^{1/3}\log x)$ in 1904. G.H. Hardy \cite{gh1,gh2} showed that $\Delta (x)=\mathcal{O}(x^\theta)$ is not possible for $\theta<1/4$. The exponent $1/3$ is important  in  that it can reached by relatively easy methods, and yet is very difficult to significantly improve. Indeed after a century of effort by many illustrious mathematicians, mainly using exponetial sum methods, the current best  result is due to 
	 M. N. Huxley \cite{hux3} proving  $\Delta (x)=\mathcal{O}(x^{\theta}\log^{\eta} x)$ 
	with $\theta={131}/{416},$ and     $\eta=1+{18627}/{8320}.$

	   When $d=2$,  the linear transformation $(a,b)=T(j,k)=(k-j,k+j)$ makes  estimating \eqref{os} essentially equivalent to the Dirichlet divisor problem, see \cite{khl1}.  For $d\geq 3$ and $x>0$  the sets 
	\begin{equation*}
		R_d(x):=|\{ (j,k)\in \Z^2:  0<|k|^d-|j|^d\leq x    \}|,
	\end{equation*} 
	have been studied as higher power analogues of the Dirichlet divisor problem, and of course  \eqref{os} can be estimated from these via symmetry.
	 We have by  Kratzel \cite{krat2} and Nowak \cite{nowak1,nowak2}
	\begin{equation}\label{R-}
		R_d(x)=A_d x^{2/d}+B_d x^{1/(d-1)}+D_dF_d(x^{1/d})x^{1/d-1/d^2}+ \Delta_d(x),
	\end{equation}
	where   the first term  comes from the area,  the second from the length of the boundary, and their coefficients are explicitly known constants that depend only on $d$. 
	The $x^{1/d-1/d^2}$ term  emerges from the  zero curvature points of  the curve, and these do not exist when $d=2$. The coefficient $D_d$  is another explicitly known constant depending only on $d$, however $F_d$ is a bounded function of $x$
	\begin{equation*}
		 F_d(x)=\sum_{n=1}^{\infty}n^{-1-1/d}\sin\Big(2\pi nx+\frac{\pi}{2d}\Big),  \quad   	|F_d(x)|\leq\sum_{n=1}^{\infty}n^{-1-1/d}.
	\end{equation*}
 Thus  the asymptotic expansion for $R_d(x)$ cannot be developed beyond this term, and this term is $\Oh(x^{1/d-1/d^2})$.   The error term $\Delta_d(x)$ is  analogous to the error term of the Dirichlet divisor problem, and 
	applying the results of Huxley \cite{hux1,hux3}, Nowak \cite{nowak1,nowak2} obtained 
	\begin{equation}\label{nwkr}
		\Delta_d(x) =\Oh(x^{\frac{46}{73}\frac{1}{d}+\varepsilon}), \qquad
	\Delta_d(x) =\Omega(x^{\frac{1}{2d}}).
	\end{equation}

With the help of   these results on $R_d(x)$ we can  deduce  the next theorem on  \eqref{ss2}. But  we will also  show that this theorem can be proven directly without recourse to viewing the set as a margin.
	\begin{theorem}\label{at1}
		Let $E,D\geq 1$ be real numbers, and let $d\geq 3$ be an integer. Then
		\begin{equation}\label{p6l21}
			\sup_{D\leq E\leq D^2}	\big|\big\{ (j,k) \in \N^2: j<k,  \   |k^d-j^d-E| <D \big\}\big|\leq C_d D^{2/d}.
		\end{equation}
		\end{theorem}
	From  the literature on lattice point problems, see for example \cite{gk}, we know that bounds of order $\Oh(x^{\frac{1}{d}})$ are reachable without any recourse to  exponential sums, merely by carefully  decomposing the region, and  judiciously   choosing the direction of the lines with which we  count the  lattice points. We will show  that the same is true for  our Theorem 8. However the proof is considerably long and difficult.

	Let us also look at a special case of this. In this case fixing $E=D^s, \ 1<s\leq 2$ gives some gain in the exponent.
	
	\begin{theorem}\label{at2}
		Let $D\geq 1$ be a real number, and $1<s\leq 2$. Let $d\geq 3$ be an integer. Then	
		
		\begin{equation}\label{p6sl3}
			\big|\big\{ (j,k) \in \N^2: j<k,  \   |k^d-j^d-D^s| <D \big\}\big|\leq \begin{cases} C_d D^{1+s(\frac{2}{d}-1)}  \ \ &\text{if} \ \   1<s\leq \frac{d^2}{d^2-d-1},       \\   C_d D^{\frac{s}{d}(1-\frac{1}{d})}  \ \ &\text{if} \ \  \frac{d^2}{d^2-d-1} \leq s\leq 2.  \end{cases}
		\end{equation}
			\end{theorem}

The arithmetic problems mentioned here concern regions bounded by hyperbolic curves, and  they have  elliptic analogues in the Gauss circle problem and its higher power variants. Essentially the same methods apply to   both class of problems.  Also analogous problems are studied in $\R^n$ for any $n\geq2$. For excellent expositions of all aspects of  this topic  see the books \cite{gk,hux2,krat} and the survey \cite{ikkn}.

	Our investigations in this article  usher  new directions of research, and  raise many further questions.
	For example it would be interesting to prove  analogues of Theorems 2,3 for arbitrary sets, or to  prove Theorem 7 for processes with less variation. But the most interesting question is passing from these random sums to  deterministic sums.  Observe that
	\begin{equation*}
		\begin{aligned}
	{\bf I}=	\Big\|\sum_{j\in A} e^{2\pi i yj^d}\Big\|_{2n}^{2n}\lesssim 	\mathbb{E}\Big\|\sum_{j\in A} e^{2\pi i yN(j^d)}\Big\|_{2n}^{2n} +	\mathbb{E}\Big\|\sum_{j\in A} e^{2\pi i yj^d}-e^{2\pi i yN(j^d)}\Big\|_{2n}^{2n}={\bf II+III}.
		\end{aligned}
		\end{equation*}
	Since ${\bf III}\lesssim {\bf I}+{\bf II}$ is plain, our results on ${\bf II}$ imply that ${\bf III}$ is dominated by the conjectural bounds on ${\bf I}.$ But is it  any easier to demonstrate 
	these for ${\bf III}$  than for ${\bf I}$?

	The rest of the article is structured as follows. We first  give a preliminaries section in which we  introduce the stochastic processes we use, and prove various lemmas on their properties that will be needed later on.  Then we  prove Theorems 1-7 each in its own section. Our final section will prove the last two theorems on  arithmetic.
	

	\section{Preliminaries}
	
	We start with briefly introducing the three types of processes that we use.  As we will make frequent use of their definitions and basic  properties,  it will be useful to the reader to have them under hand.  We assume all processes in the article to be integer valued, ensuring periodicity for the exponential sums. A  process $\{X_t\}_{t\in I}$, where $I$ is an index set, is  stationary  if joint cumulative distribution functions remain the same under time translation.  This of course means that each $X_t$ has the same distribution, and this will be the only property that we will use to prove Theorem 1.  
	
	A random variable $N$ has Poisson distribution if there exist $\mu>0$ such that for any nonnegative integer $k$
	\begin{equation}\label{pd}
		\Pp[N=k]=e^{-\mu}\frac{\mu^k}{k!}.
	\end{equation}
	Then $N$ has mean and variance $\mu$.
	A process $\{N(t)\}_{t\geq 0}$ is a Poisson process with intensity $\lambda$ if three conditions hold:
	
	{\bf 1.} $N(0)=0$ almost surely.
	
	{\bf 2.} $N(t)$ has independent increments, that is for $0\leq t_1\leq t_2\leq \ldots\leq t_k$ the random variables $N(t_1),N(t_2)-N(t_1),\ldots, N(t_k)-N(t_{k-1})$ are independent.
	
   	{\bf 3.} For any $0 \leq s < t < \infty$ the random variable $N(t)-N(s)$ is a Poisson random variable 
   with mean $\lambda(t - s).$\\
		In this article we will always take the intensity $\lambda$ to be 1.
	
	Let $X_j, j\in \N$ be independent identically distributed Bernoulli variables with $\Pp[X_j=1]=\Pp[X_j=-1]=1/2$. Then simple random walk $\{R(n)\}_{n\in\Z_+}$ is the process given by 
	\[ R(n)=\sum_{j=1}^n X_j  \]
	with the convention that $R(0)=0.$
	This is  a mean zero process with $\V[R(n)]=n.$
	For more information on these processes see any basic probability book such as \cite{ks}.

	Our  lemmas on the Poisson process aim to make rigorous the heuristic that  for an  intensity 1  Poisson process  $\{N(t)\}_{t\geq 0}$,  the random variable $N(j)$ only takes the integer values between $j-\sqrt{j},j+\sqrt{j}$ each with equal probability. With the lemma below, we will carry out the first part of this, that is  $N(j)$ only takes the integer values between $j-\sqrt{j},j+\sqrt{j}$.   Essentially this is because the Poisson process  is an independent increment process, which leads to Gaussian-like concentration around the mean. As such similar results can be proved for more general classes of processes. More words about this after the proof.
	\begin{lemma}\label{lem1}
		Let $\{N(t)\}_{t\geq 0}$ be a Poisson process of intensity 1, and let $m\in \mathbb{N}$. Let $0<\lambda \leq \sqrt{m}.$ Then 
		\begin{equation*}
			\mathbb{P}[|N(m)-m|>\lambda \sqrt{m}]\leq 2e^{-\lambda^2/4}.
		\end{equation*}

	\end{lemma}

	\begin{proof}
		We start with
		\begin{equation*}
			\begin{aligned}
				\mathbb{P}[|N(m)-m|>\lambda \sqrt{m}]=\mathbb{P}[N(m)-m>\lambda \sqrt{m}]+\mathbb{P}[m-N(m)>\lambda \sqrt{m}]= {\bf I}+{\bf II}.
			\end{aligned}
		\end{equation*}
		We first deal with  {\bf I}. Let $t>0$.
		\begin{equation*}
			\begin{aligned}
				{\bf I}=\mathbb{P}[t(N(m)-m)>t\lambda \sqrt{m}]=\mathbb{P}[e^{t(N(m)-m)}>e^{t\lambda \sqrt{m}}] &=\mathbb{P}[e^{t(N(m)-m-\lambda \sqrt{m})}>1]\\ &\leq \mathbb{E}[e^{t(N(m)-m-\lambda \sqrt{m})}] \\ &= e^{-t\lambda \sqrt{m}-tm}\mathbb{E}[e^{tN(m)}].
			\end{aligned}
		\end{equation*}
		The last expectation is of course the moment generating function(MGF) of $N(m)$, and can easily be calculated from \eqref{pd} to be $e^{m(e^t-1)}.$ Notice then that this is just the $m$th power of the MGF of $N(1)$, and this is precisely because of the independent increment property.
		Observe that  for $0<t\leq 1$
			\begin{equation*}
			\begin{aligned}
			e^t=1+t+\frac{t^2}{2}\Big[1 +\frac{t}{3}+\frac{t^2}{4\cdot 3}+\cdots   \Big]\leq 1+t+\frac{t^2}{2}\Big[1 +\frac{1}{3}+\frac{1}{ 3^2}+\cdots   \Big] \leq  1+t+\frac{3}{4}t^2.
			\end{aligned}
		\end{equation*}
		 Plugging in the MGF of $N(m)$ and then using this  yields
		\begin{equation*}
	{\bf I}\leq e^{-t\lambda\sqrt{m}+m(e^t-t-1)}\leq  e^{-t\lambda\sqrt{m} +\frac{3}{4}mt^2  },\qquad  \qquad     0<t\leq 1.
		\end{equation*}
		Here choosing $t=\lambda/\sqrt{m}$ bounds ${\bf I}$ by  $e^{-\lambda^2/4}.$

		Exactly the same steps give 
		\begin{equation*}
			\begin{aligned}
				{\bf II}\leq e^{-t\lambda\sqrt{m}+m(e^{-t}-1+t)}.
			\end{aligned}
		\end{equation*}
	Observe that  for $0<t\leq 1$ 
	\begin{equation*}
		\begin{aligned}
			e^t=1-t+\frac{t^2}{2}-\Big[\frac{t^3}{3!}-\frac{t^4}{4!}\Big]-\Big[\frac{t^5}{5!}-\frac{t^6}{6!}\Big]\cdots   \leq 1-t+\frac{t^2}{2},
		\end{aligned}
	\end{equation*}
as each of the parantheses above is positive.
	From this  
		\begin{equation*}
			\begin{aligned}
				{\bf II}\leq e^{-t\lambda\sqrt{m}+m\frac{t^2}{2}}.
			\end{aligned}
		\end{equation*}
		Choosing  $t=\lambda/\sqrt{m}$ gives the bound ${\bf II}\leq e^{-\lambda^2/2}$, finishing the proof.

	\end{proof}

	If a random variable $X$ for some positive constant $K_1$ satisfies 
	\[ \Pp[|X|\geq t] \leq 2e^{-t^2/K^2_1}\]
	for all $t\geq 0$, it is called subgaussian.  An equivalent definition is that $X$ is subgaussian if  
	\[ \E[e^{tX}] \leq e^{K^2_2t^2}\]
	for all $t\in \R$ and for some constant $K_2.$ So our random variables  $N(m)-m$ behave like subgaussians, but strictly speaking they are not subgaussians as is plain from their moment generating functions $e^{m(e^t-t-1)}$. But for a restricted range of $t$  we see the relevant inequalities satisfied. In \cite{cck} the authors  coined the term \textquotesingle locally subgaussian\textquotesingle \  for  such random variables.

	The next two lemmas handle the rest of the heuristic we presented at the start of the section. They  complement each other nicely.
	\begin{lemma}\label{l3}
		Let $a\in \N$. Then
		\begin{equation*}
			\sup_{t\geq 0} \Pp[N(t)=a]\leq \frac{1}{\sqrt{2\pi a}}.
		\end{equation*}
	\end{lemma}

	\begin{proof}
		We have  
		\begin{equation*}
			\sup_{t\geq 0}\Pp[N(t)=a]=	\sup_{t\geq 0}   e^{-t}\frac{t^a}{a!}=\frac{1}{a!}\sup_{t\geq 0}   t^ae^{-t}.
		\end{equation*}
		Plainly we are seeking the supremum of a smooth nonnegative function on $[0,\infty)$, and this function vanishes at the endpoints of this interval. So to find the supremum  we differentiate and set it equal to zero. 
		\begin{equation*}
			0=\frac{d}{dt}  \big(t^ae^{-t}\big)=at^{a-1}e^{-t}-t^ae^{-t}=t^{a-1}e^{-t}(a-t).
		\end{equation*}
		Thus supremum is attained at $t=a,$ and is $a^ae^{-a}.$
		   A precise version of Stirling's formula due to Robbins \cite{rob} states that  for any $n\in \N$
		\begin{equation}\label{robs}
			\frac{1}{\sqrt{2\pi n}}e^{-\frac{1}{12n}}\leq  \frac{1}{n!}\cdot \frac{n^n}{e^n}\leq \frac{1}{\sqrt{2\pi n}}e^{-\frac{1}{12n+1}}\leq \frac{1}{\sqrt{2\pi n}}.
		\end{equation}
		This concludes the proof.
		
	\end{proof}

	\begin{lemma}\label{l4}
		Let $t\geq 0$.  Then 
		\begin{equation*}
			\sup_{a\in \mathbb{Z}_+}\Pp[N(t)=a]=\Pp[N(t)=\lfloor t\rfloor ]\leq \min\big\{1,\frac{1}{\sqrt{2\pi \lfloor t\rfloor}  }\big\}.
		\end{equation*}
	\end{lemma}
	
	\begin{proof}
		We observe that as
		\[ \Pp[N(t)=a]=e^{-t}\frac{t^a}{a!} \]
		the supremum over $a$ depends only on the fraction $t^a/a!$ As 
		\[\frac{t^{a+1}}{a+1!}=\frac{t}{a+1}\frac{t^a}{a!},\]
		as long as $a+1\leq \lfloor t\rfloor $ this fraction increases, and then for $a+1>\lfloor t\rfloor $ it decreases. That is, the last multiplication must be by $t/\lfloor t\rfloor$, with which we obtain $t^{\lfloor t\rfloor }/\lfloor t\rfloor !$. On finding this we apply Lemma \ref{l3} to conclude the proof.	
	\end{proof}

To calculate probabilities  and reduce to a lattice point problem in the general $p=2n$ case we need a number of lemmas.
We start with  the following well known fact on independent random variables. For a proof see Sinai-Koralov \cite{ks}.

\begin{lemma}
	Let $X_1,\ldots, X_n$ be independent random variables, $m_1+\ldots+m_k =
	n$ and $f_1,\ldots, f_k$ be Borel measurable functions of $m_1, ...,m_k$ variables respectively.
	Then the random variables $ f_1(X_1, ..., X_{m_1} ),  f_2(X_{m_1+1}, \ldots, X_{m_1+m_2} ),\ldots,
	f_k(X_{m_1+\ldots+m_{k-1}+1},\ldots, X_n)$ are independent.

\end{lemma}
Our next lemma is crucial for  calculating probabilities. With this lemma we separate the random variables  into two groups such that the sums over these groups are independent. After this independence is utilized next lemmas
can be brought in to calculate probabilities.

\begin{lemma}
	Let $\{X_t\}_{t\geq 0}$ be an independent increment process. 	Let $n\geq 2$ be an integer, and let $(j_i,k_i), \  1\leq i\leq n$ be nonempty open intervals.  Let $\sigma_1\cup\sigma_2$ be a partition of indices $\{1,2,\dots,n\}$ into  two nonempty subsets.  Suppose no interval indexed by  $\sigma_1$ intersects the intervals indexed by $\sigma_2$ and vice versa. Then the random variables 
	
	\begin{equation*}
		\sum_{i\in \sigma_1} X_{k_i}-X_{j_i}, \qquad   \sum_{i\in \sigma_2} X_{k_i}-X_{j_i}
	\end{equation*}
	are independent.	
\end{lemma}

\begin{proof}
	Let us list $j_i,k_i, \ i\in \sigma _1$  with respect to size, and let $t_1\leq t_2\leq\ldots\leq t_{2|\sigma_1|}$ be this list. For the intervals $(t_i,t_{i+1}), \  1\leq i \leq 2|\sigma_1|-1,$ let $ i\in \eta_1\subseteq\{1,2,\ldots,2|\sigma_1|-1 \}$
	be those indices for which $(t_i,t_{i+1})$ lies inside one of $(j_i,k_i), \ i\in \sigma_1$. Then observe that 
	\begin{equation*}
		\sum_{i\in \sigma_1} X_{k_i}-X_{j_i}=	\sum_{i\in \eta_1}d_i(X_{t_{i+1}}-X_{t_{i}})
	\end{equation*}
	for some natural numbers $d_i.$ 	We apply this same procedure to indices in $\sigma_2$ to obtain  intervals $(t_i,t_{i+1}), \ i\in \eta_2$.  Then we observe that the intervals indexed by  $\eta_1,\eta_2$ are all disjoint. Therefore bearing in mind the independent increment property we can apply Lemma 4 to obtain our claim.

\end{proof}		

The next two lemmas allow explicit computation of probabilities in  Theorem 7, as we will need an analogue of Lemmas 2,3 for a linear combination of Poisson random variables. Proofs will follow from the independence of these variables and Lemma 3.	

\begin{lemma}\label{l5}
	Let $N_i$ be  independent random variables of Poisson distribution with mean $\mu_i> 0$, and $d_i\in \N$  where $1\leq i\leq n.$ Let $\mu:=\max_{1\leq i\leq n}\mu_i$. Then  for $a\in \Z_+$
	\begin{equation*}
		\Pp\big[\sum_{i=1}^n d_iN_i=a \big]\leq \min\big\{1, \frac{1}{\sqrt{2\pi\lfloor\mu \rfloor}}  \big\}.
	\end{equation*}
	
\end{lemma}

\begin{proof}
	Without loss of generality assume  $\mu_n=\mu$. We observe 
	\begin{equation*}
		\Pp\big[\sum_{i=1}^n d_iN_i=a \big]=	\sum_{\substack{a_1+a_2+\ldots+a_n=a \\  a_i\in \Z_+}}\Pp\big[ N_i=a_i/d_i, \  1\leq i\leq n \big]=\sum_{\substack{a_1+a_2+\ldots+a_n=a \\  a_i\in \Z_+}}\prod_{i=1}^n \Pp\big[ N_i=a_i/d_i]    .
	\end{equation*}
	Now we can use the assumption that $\mu_n=\mu$  together with Lemma 3 and reduce the number of factors in the product by one. 
	\begin{equation*}
		\begin{aligned}
			\leq\min\big\{1, \frac{1}{\sqrt{2\pi\lfloor\mu \rfloor}}  \big\}\sum_{\substack{a_1+a_2+\ldots+a_n=a \\  a_i\in \Z_+}}\prod_{i=1}^{n-1} \Pp\big[ N_i=a_i/d_i].
		\end{aligned}
	\end{equation*}
	This opens up the possibility of summing over $a_i,\  1\leq i\leq n$.
	\begin{equation*}
	\begin{aligned}
		\leq \min\big\{1, \frac{1}{\sqrt{2\pi\lfloor\mu \rfloor}}  \big\} \sum_{\substack{a_i=0 \\ 1\leq i\leq n-1 }}^a \prod_{i=1}^{n-1} \Pp\big[ N_i=a_i/d_i]   \leq 
		\min\big\{1, \frac{1}{\sqrt{2\pi\lfloor\mu \rfloor}}  \big\}\prod_{i=1}^{n-1}  \sum_{a_i=0 }^a \Pp\big[ N_i=a_i/d_i]. 
	\end{aligned}
\end{equation*}
	As each of the sums in the last term are at most 1, this concludes the proof.	
\end{proof}

This more general lemma will be used in the following specific form  in proving Theorem 7.

\begin{lemma}
	Let $(j_i,k_i), \ 1\leq i\leq n$ be nonempty open intervals, and let $\{N(t)\}_{t\geq0}$ be a Poisson process of intensity 1. Let $m$ be an index at which $k_i-j_i$ becomes maximum. Then for any $a\in \Z_+$
	\begin{equation*}
		\Pp\big[  	\sum_{i=1}^n N(k_i)-N(j_i)=a    \big]\leq \min\big\{1,\frac{1}{\sqrt{2\pi\lfloor (k_m-j_m)/2n\rfloor}}\big\}.
	\end{equation*}
	
\end{lemma}	

\begin{proof}
	We apply the procedure described in Lemma 5 to obtain the disjoint intervals $(t_i,t_{i+1}), \ i\in \eta$ for a set of indices $ \eta \subseteq\{1,2,\ldots2n-1 \}$ that enable writing
	\begin{equation*}
		\sum_{i=1}^{n} N(k_i)-N(j_i)=	\sum_{i\in \eta}d_i[N(t_{i+1})-N(t_{i})]
	\end{equation*}
	for some natural numbers $d_i.$ Since the union of $(t_i,t_{i+1}), i\in \eta$ is the same as the union of $(j_i,k_i), \ 1\leq i\leq n$, at least one  $(t_i,t_{i+1}), i  \in \eta$ has length at least  $(k_m-j_m)/2n.$
	An application of Lemma 6 concludes the proof.
	
\end{proof}

	The next two lemmas are technical and elementary in nature. 
	They will be needed at certain steps of the proofs of Theorems 4,7.
	
	\begin{lemma}\label{dl}
		Let $A\subseteq \N$ be a finite set, and let $a\in \R, \ b\in \N$ satisfy  $0<a\leq b\leq \min A$. Let  $d\geq 2$ be an integer.   Let $\Phi:[a,\infty)\rightarrow (0,\infty)$ be a decreasing function. Then
		
		\begin{equation*}
			\sum_{\substack{j,k\in A  \\  j<k  }  }\Phi(k^d-j^d)  \leq 	\sum_{\substack{b\leq j<k\leq b+ |A|-1  }  }\Phi(k^d-j^d).  
		\end{equation*}
		
	\end{lemma}
	
	\begin{proof}
		We list the elements of $A$ into a strictly  increasing sequence $a_1,a_2,\ldots,a_{|A|}.$
		Then for any $1\leq j<k\leq |A|$ we have 
		$a_{j+1}-a_j\geq 1$, implying by telescoping $a_k-a_j\geq k-j.$   Choosing $j=1$, and bearing in mind that $a_1=\min A$, we obtain $a_k\geq b+k-1.$ Combining these
		
		\begin{equation*}
			\begin{aligned}
				a_k^d-a_j^d=\big(a_k-a_j\big)\sum_{i=0}^{d-1}a_k^{d-1-i}a_j^{i}  &\geq \big(k-j\big)\sum_{i=0}^{d-1}(b+k-1)^{d-1-i}(b+j-1)^{i}\\  &=(b+k-1)^d-(b+j-1)^d. 
			\end{aligned}
		\end{equation*}
	Since $d\geq 2$, the last term is not less than $2b$, so both the first and last term lies in the domain of $\Phi.$	As $\Phi$ is decreasing, this allows us to conclude
		\begin{equation*}
			\begin{aligned}
				\sum_{\substack{j,k\in A  \\  j<k  }  }\Phi(k^d-j^d)=	\sum_{\substack{1\leq j,k\leq |A|  \\  j<k  }  }\Phi(a_k^d-a_j^d)&\leq \sum_{\substack{1\leq j<k\leq |A|    }  }\Phi((b+k-1)^d-(b+j-1)^d)  \\ &=\sum_{\substack{b\leq j<k\leq b+|A|-1    }  }\Phi(k^d-j^d).  
			\end{aligned}
		\end{equation*}
	\end{proof}

\begin{lemma}\label{lem2}
	Let  $1\leq C\leq 10$.  
	
	{\bf a.} If $x>e^{50}$, then 
	\begin{equation*}
		\begin{aligned}
			|x-y|\leq C\sqrt{x\log x} \ \  \implies  \ \ |x-y|\leq 2C\sqrt{y\log y}. 
		\end{aligned}
	\end{equation*}

	{\bf b.} If $y>e^{50},x\geq 1$ then
	\begin{equation*}
		\begin{aligned}
			|x-y|\geq 2C\sqrt{x\log x} \ \  \implies \ \
			|x-y|\geq {C}\sqrt{y\log y}.
		\end{aligned}
	\end{equation*}
	
\end{lemma}

\begin{proof}
	To see {\bf a.} observe that $y\geq x-C\sqrt{x\log x}\geq x/2$, thus $x\leq 2y$. 
	As $\sqrt{x\log x}$ is an increasing function, we have $\sqrt{x\log x} \leq \sqrt{2y\log 2y}\leq 2\sqrt{y\log y}$, and we are done.
	
	As for {\bf b.}, if  $y\leq 2x$,  owing to the same reason $\sqrt{y\log y} \leq \sqrt{2x\log 2x}\leq 2\sqrt{x\log x}$, and we are done. If on the other hand $y\geq 2x$, then $y-x\geq y/2$, and since 
	${C}\sqrt{y\log y}\leq y/2$, we are done.

\end{proof}

	
	\section{Stationary processes}

	This section is dedicated to proving Theorem 1. About stationary processes we will only use that the distribution remains the same in time. From harmonic analysis and probability we will use  the techniques that are the mainstay of studying random and deterministic exponential sums.

	\begin{proof}
		Let the probability mass function of the  random variables in our process be  given by $\mu_{X_j}(k)=\mu(k)=\mu_k.$ 
		We will see that \eqref{s2} and \eqref{s3} follow from   \eqref{s1}, and for this we first need to investigate the cases $p=2,\infty$. We start with the latter. Clearly for every $\omega\in \Omega$ by considering a neighborhood of $y=0$
		we have 
		
		\begin{equation}\label{sp1}
			\Big\|\sum_{j\in A}e^{2\pi i yX_j(\omega)}\Big\|_{\infty}=|A|. 
		\end{equation}
		From this  all claims of the theorem for $p=\infty$  follow immediately. So let $1\leq p< \infty$. Also immediate from \eqref{sp1}  are 
		\begin{equation}\label{sp-1}
			\Big\|\sum_{j\in A}e^{2\pi i yX_j(\omega)}\Big\|_{p}\leq |A|, \qquad	\E \Big\|\sum_{j\in A}e^{2\pi i yX_j}\Big\|_{p}^p\leq |A|^p. 
		\end{equation}
		So we only need to show the reverse direction to obtain \eqref{s1}. 
		
		For $p=2$, to this end   we write
		\begin{equation}\label{sp2}
			\begin{aligned}
				\mathbb{E}\Big\|\sum_{j\in A} e^{2\pi i yX_j}\Big\|_{2}^2 =\mathbb{E}\int_{\mathbb{T}}\Big|\sum_{j\in A} e^{2\pi i yX_j}\Big|^2dy&=\mathbb{E}\int_{\mathbb{T}}\sum_{j,k\in A} e^{2\pi i y(X_k-X_j)}dy\\ &= \mathbb{E}\sum_{k\in A}\sum_{j\in A}\int_{\mathbb{T}} e^{2\pi i y(X_k-X_j)}dy.
			\end{aligned}
		\end{equation}
		For  a fixed $\omega$ the integral in the last expression is $1$ if $X_j(\omega)=X_k(\omega)$ and $0$ otherwise. We define  
		
		\begin{equation*}
			U(\omega,k):=|\{ j\in A: X_j(\omega)=k\}|=\sum_{j\in A} \mathbb{I}_{\{X_j=k\}}(\omega),
		\end{equation*}
		where $\mathbb{I}$ denotes the indicator function. For each $k$, as $U(\omega,k)$ is a finite sum of indicator functions it is measurable.  Therefore we have
		\begin{equation*}
			\mathbb{E}U(\omega,k)=\sum_{j\in A} \mathbb{E}\ \mathbb{I}_{\{X_j=k\}}(\omega)=\sum_{j\in A} \mathbb{P}[X_j=k]=|A|\mu_k.
		\end{equation*}
		The Cauchy-Schwarz inequality yields $\mathbb{E}U^2(\omega,k)\geq [\mathbb{E}U(\omega,k)]^2=|A|^2 \mu_k^2$, and applying this we  continue from  \eqref{sp2} as follows to obtain \eqref{s1} for  $p=2$
		
		\begin{equation*}
			= \mathbb{E}\sum_{j\in A} U(\omega,X_j(\omega))=\mathbb{E}\sum_{k\in \mathbb{Z}}U^2(\omega,k) =\sum_{k\in \mathbb{Z}}\mathbb{E}U^2(\omega,k)\geq |A|^2\sum_{k\in \mathbb{Z}} \mu_k^2.
		\end{equation*}
		
		For $2<p<\infty$  we utilize the  $p=2,\infty$ cases:
		\begin{equation*}
			\mathbb{E}\Big\|\sum_{j\in A}e^{2\pi i yX_j}\Big\|_{p}^p \geq  \mathbb{E}\Big\|\sum_{j\in A} e^{2\pi i yX_j}\Big\|_{2}^p  \geq \Big[ \mathbb{E}\Big\|\sum_{j\in A} e^{2\pi i yX_j}\Big\|_{2}^2\Big]^{p/2} \geq |A|^p \Big[\sum_{k\in \mathbb{Z}} \mu_k^2\Big]^{p/2}.
		\end{equation*}
		For $1\leq p<2$  we again utilize the same ideas:
		\begin{equation*}
			\begin{aligned}
				\mathbb{E}\Big\|\sum_{j\in A} e^{2\pi i yX_j}\Big\|_{2}^2  \leq \mathbb{E}     
				\Big\|\sum_{j\in A} e^{2\pi i yX_j}\Big\|_{\infty}^{2-p}
				\int_{\mathbb{T}}\Big|\sum_{j\in A} e^{2\pi i yX_j}\Big|^{p}dy  = |A|^{2-p}\mathbb{E}\Big\|\sum_{j\in A} e^{2\pi i yX_j}\Big\|_{p}^p. 
			\end{aligned}
		\end{equation*}
		Therefore
		
		\begin{equation*}
			\mathbb{E}\Big\|\sum_{j\in A} e^{2\pi i yX_j}\Big\|_{p}^p \geq |A|^p \sum_{k\in \mathbb{Z}} \mu_k^2 . 
		\end{equation*}
		finishing  the proof of \eqref{s1}.
		
		To see \eqref{s2}
		we just need to observe that for any $\omega$ and any sequence in $\{a_j\}_{j\in A}, \ |a_j|\leq 1$
		\begin{equation*}
			\Big\|\sum_{j\in A}a_j e^{2\pi i yX_j(\omega)}\Big\|_{\infty}\leq |A|, \qquad 
			\text{and}  \qquad
			\Big\|\sum_{j\in A}a_j e^{2\pi i yX_j(\omega)}\Big\|_{p}^p\leq |A|^p. 
		\end{equation*}
		We use this to conclude that for $1\leq p<\infty$
		\begin{equation*}\label{}
			\mathbb{E}\sup_{ |a_j|\leq 1}\Big\|\sum_{j\in A}a_j e^{2\pi i yX_j}\Big\|_{p}^p\leq |A|^p, \ \ \ \ \ \   \mathbb{E}\sup_{  |a_j|\leq 1}\Big\|\sum_{j\in A} a_j e^{2\pi i yX_j}\Big\|_{\infty}\leq |A|. 
		\end{equation*}
		Then using \eqref{s1} we obtain \eqref{s2}.

		To handle \eqref{s3} we need to introduce another method.  Let $\epsilon >0$, and let $K$ be such that 
		\begin{equation*}
			\sum_{|k|\leq K} \mu_k >1-\epsilon.
		\end{equation*} 
		Therefore 
		\begin{equation*}
			\mathbb{E}\sum_{j\in A}\mathbb{I}_{\{|X_j|\leq K\}}\geq (1-\epsilon)|A| .
		\end{equation*}
		Then for  the set
		\begin{equation*}
			\Omega_{\epsilon}:=\{\omega\in \Omega: \sum_{j\in A}\mathbb{I}_{\{|X_j|\leq K\}} \geq \frac{9}{10}|A|  \} 
		\end{equation*}
		we have $\mathbb{P}(\Omega_{\epsilon})\geq 1-10\epsilon$.  We observe that for $\omega\in	\Omega_{\epsilon}$ and $0 \leq y\leq   1/100K$ we have
		\begin{equation*}
			\begin{aligned}
				\Big|\sum_{j\in A}e^{2\pi i yX_j(\omega)}\Big|&\geq \Big|\sum_{ |X_j(\omega)|\leq K} e^{2\pi i yX_j(\omega)}\Big|-\Big|\sum_{|X_j(\omega)|> K} e^{2\pi i yX_j(\omega)}\Big|\\  &\geq \Big|\sum_{ |X_j(\omega)|\leq K} \cos{2\pi  yX_j(\omega)}\Big|-\frac{|A|}{10}\\  &\geq \frac{9|A|}{20}-\frac{|A|}{10} >\frac{|A|}{4}.
			\end{aligned}
		\end{equation*}  
		For any  $\omega\in	\Omega_{\epsilon}$, and any $1\leq p< \infty$ we apply this last result to obtain
		\begin{equation*}\label{}
			\Big\|\sum_{j\in A}e^{2\pi i yX_j(\omega)}\Big\|_{p}^p\geq \int_{0}^{{1}/{100K}}\Big|\sum_{j\in A} e^{2\pi i yX_j(\omega)}\Big|^pdy \geq \frac{|A|^p}{4^p100K}.
		\end{equation*}
		Thus we understand that 
		\begin{equation*}
			\mathbb{P}\Big[ \Big\|\sum_{j\in A} e^{2\pi i yX_j(\omega)}\Big\|_{p}^p \geq \frac{|A|^p}{4^p100K} \Big]\geq 1-10\epsilon.
		\end{equation*}
		Therefore for $|A|$ larger than a constant depending only on $\epsilon,\varepsilon, \mu,p$
		\begin{equation*}
			\mathbb{P}\Big[|A|^{\varepsilon} \Big\|\sum_{j\in A} e^{2\pi i yX_j(\omega)}\Big\|_{p}^p \geq |A|^{p+\varepsilon/2} >\sup_{S_A}  \Big\|\sum_{j\in A} a_j e^{2\pi i yX_j(\omega)}\Big\|_{p}^p \Big]\geq 1-10\epsilon.
		\end{equation*}
		Taking limits and considering that  $\epsilon$ is  arbitrary we obtain our result.

	\end{proof}
	
	With the extra assumption that $X_j$ are independent random variables it would be easier to compute \eqref{sp2}. In that case we would bring the expectation inside, and then compute the probabilities exactly. In proving some of our  other theorems this idea will be used in a modified form.


	\section{Poisson processes}

In this section  we prove an analogue of Theorem 1 for Poisson processes. Unfortunately  in this case we have to consider the set $A$ to be of interval form, for the analogue of \eqref{s1} in this case does not hold, that is the left hand side  depends not only on the cardinality but also on the structure of the set.  Proving analogues of \eqref{s2},\eqref{s3} without using an estimate like \eqref{s1} is  very difficult.  As the proof is similar to that of Theorem 1, we  present it briefly.

	\begin{proof}
		As before, the analogues of  \eqref{s2} and \eqref{s3} will follow from   \eqref{p1}, and for this we first need to investigate the cases $p=2,\infty$. We start with the latter. Clearly for every $\omega\in \Omega$ by considering a neighborhood of $y=0$
		we have 
		
		\begin{equation*}\label{pp11}
			\Big\|\sum_{j=1}^M e^{2\pi i yN(j)(\omega)}\Big\|_{\infty}=M. 
		\end{equation*}
		From this all  claims of the theorem for $p=\infty$  follow immediately. So we assume $2\leq p< \infty$.

		For \eqref{p1} with $p=2$ we compute
		\begin{equation*}\label{}
			\begin{aligned}
				\mathbb{E}\Big\|\sum_{j=1}^M  e^{2\pi i yN(j)(\omega)}\Big\|_{2}^2= \mathbb{E} \sum_{j,k=1}^M \int_{\mathbb{T}} e^{2\pi i y[N(j)(\omega)-N(k)(\omega)]}dy&= \sum_{j,k=1}^M\mathbb{P}[N(|j-k|)=0]
				\\ &= \sum_{j,k=1}^Me^{-|j-k|},
			\end{aligned}
		\end{equation*}
		and 
		\begin{equation*}
			M\leq \sum_{j,k=1}^Me^{-|j-k|}=M+ 2\sum_{j=1}^{M-1}(M-j)e^{-j}\leq M+ 2M\sum_{j=1}^{M-1}e^{-j}\leq 3M.
		\end{equation*}
		This finishes \eqref{p1} for $p=2$. Then we extend this to $p>2$. For any sequence $\{a_j\}_{j=1}^M, \ |a_j|\leq 1 $
		\begin{equation}\label{pp2}
				\begin{aligned}
			\Big\|\sum_{j=1}^M a_j e^{2\pi i yN(j)(\omega)}\Big\|_{2}^2 \leq \Big\|\sum_{j=1}^M  e^{2\pi i yN(j)(\omega)}\Big\|_{2}^2,
		\end{aligned}	
	\end{equation}
		therefore 
		\begin{equation}\label{pp-1}
			\begin{aligned}
				\E\sup_{|a_j|\leq 1}\Big\|\sum_{j=1}^M a_je^{2\pi i yN(j)}\Big\|_{p}^p  \leq \E   
				\sup_{|a_j|\leq 1}\Big\|\sum_{j=1}^M a_je^{2\pi i yN(j)}\Big\|_{\infty}^{p-2}
				\int_{\mathbb{T}}\Big|\sum_{j=1}^M a_je^{2\pi i yN(j)}\Big|^{2}dy\leq 3M^{p-1}.
			\end{aligned}
		\end{equation}
		Let $A_M:=\{N(M)>5M\}$. By Lemma 1, $\mathbb{P}[A_M] \leq 2e^{-4M}$. Therefore for $\omega\notin A_M$ 
		
		\begin{equation*}
			\begin{aligned}
				\Big\|\sum_{j=1}^M e^{2\pi i yN(j)(\omega)}\Big\|_{p}^p  \geq \int_{0}^{1/100M}\Big|\sum_{j=1}^M \cos 2\pi  yN(j)(\omega)\Big|^pdy \geq      
				\frac{M^{p-1}}{100\cdot 2^p}.
			\end{aligned}
		\end{equation*}
		Therefore 
		\begin{equation*}
			\begin{aligned}
				\mathbb{E}\Big\|\sum_{j=1}^M e^{2\pi i yN(j)}\Big\|_{p}^p &\geq     
					\frac{M^{p-1}}{100\cdot 2^{p+1}}.
			\end{aligned}
		\end{equation*}
		This concludes   \eqref{p1}, and the analogue of \eqref{s2} for any $p\geq 2$.
		On the other hand by \eqref{pp-1} 
		
		\begin{equation*}
			\mathbb{P}\Big[\sup_{|a_j|\leq 1} \Big\|\sum_{j=1}^M a_je^{2\pi i yN(j)(\omega)}\Big\|_{p}^p \geq 3M^{p-1}\log M \Big]\leq 1/\log M.
		\end{equation*}
		Therefore except for a set of measure at most $1/\log  M+2e^{-4M}$
		\begin{equation*}
			\begin{aligned}
				\sup_{|a_j|\leq 1} \Big\|\sum_{j=1}^M a_je^{2\pi i yN(j)(\omega)}\Big\|_{p}^p \leq  300\cdot 2^p\log M\Big\|\sum_{j=1}^M e^{2\pi i yN(j)(\omega)}\Big\|_{p}^p,    
			\end{aligned}
		\end{equation*}
		and this proves  the analogue of \eqref{s3}.

	\end{proof}


	\section{Random walks}
	
	This section is dedicated to proving Theorem 3. The proof is broadly similar to that of Theorem 2, and the   most important difference is that the simple random walk is not increasing as  opposed to the Poisson process. This is relevant in obtaining lower bounds by considering a neighborhood of the origin, and will be dealt with via Doob's martingale inequality. The other differences are computational.

	\begin{proof}
		As before we will  see the analogues of  \eqref{s2} and \eqref{s3} follow from   \eqref{r1}, and for this we  need to investigate the cases $p=2,\infty$. We start with the latter. Clearly for every $\omega\in \Omega$ by considering a neighborhood of $y=0$
		we have 
		
		\begin{equation}\label{pp12}
			\Big\|\sum_{j=1}^M e^{2\pi i yR(j)(\omega)}\Big\|_{\infty}=M. 
		\end{equation}
		From this all  claims of the theorem for $p=\infty$  follow immediately. Therefore  we assume $2\leq p< \infty$.
			For any sequence $\{a_j\}_{j=1}^M, \ |a_j|\leq 1, $  
		\begin{equation}\label{rw2}
			\begin{aligned}
				\Big\|\sum_{j=1}^M a_j e^{2\pi i yR(j)(\omega)}\Big\|_{2}^2 \leq \Big\|\sum_{j=1}^M  e^{2\pi i yR(j)(\omega)}\Big\|_{2}^2.
			\end{aligned}
		\end{equation}
	We have 
		\begin{equation*}
			\mathbb{E}\Big\|\sum_{j=1}^M  e^{2\pi i yR(j)}\Big\|_{2}^2 =\sum_{j,k=1}^M\mathbb{P}[R(|j-k|)=0]= M+2\sum_{1\leq j<k\leq M}^M\mathbb{P}[R(k-j)=0].
		\end{equation*}
		Bearing in mind that the simple random walk can be zero only at even times, the last sum equals
		\begin{equation*}
			\sum_{j=1}^{\lfloor(M-1)/2\rfloor} (M-2j) \mathbb{P}[R(2j)=0] =
			\sum_{j=1}^{\lfloor(M-1)/2\rfloor} (M-2j) \binom {2j}j 2^{-2j},
		\end{equation*}
		and this can be estimated via  Robbins's  Stirling  formula \eqref{robs}
		\begin{equation*}
			\begin{aligned}
				\leq  \sum_{j=1}^{\lfloor(M-1)/2\rfloor} (M-2j) \frac{1}{\sqrt{\pi j}} \leq M+\int_{1}^{M/2} \frac{M-2x}{\sqrt{x}}dx\leq  2M^{3/2}.
			\end{aligned}
		\end{equation*}
		
		 Combining  this estimate with  \eqref{pp12},\eqref{rw2} we obtain  for $2\leq p<\infty$,
		\begin{equation}\label{rw-1}
			\begin{aligned}
				\E\sup_{  |a_j|\leq 1}\Big\|\sum_{j=1}^M a_je^{2\pi i yR(j)}\Big\|_{p}^p  &\leq   \E \sup_{  |a_j|\leq 1}
				\Big\|\sum_{j=1}^M a_je^{2\pi i yR(j)}\Big\|_{\infty}^{p-2}\Big\|\sum_{j=1}^M a_je^{2\pi i yR(j)}\Big\|_{2}^2  
				\\ &\leq   M^{p-2}       \E 
				\Big\|\sum_{j=1}^M e^{2\pi i yR(j)}\Big\|_{2}^2         \leq 2M^{p-\frac{1}{2}}.
			\end{aligned}
		\end{equation}
		We now want to obtain a lower bound. As opposed to the Poisson process, the simple random walk is not increasing, therefore it is not enough to consider  $R(M)$. We turn around this difficulty using the fact that the random walk is  a martingale. By Doob's martingale inequality
		\begin{equation*}
			\mathbb{E}\sup_{1\leq j \leq M}|R(j)|^2 \leq 4\mathbb{E}|R(M)|^2 =4M.
		\end{equation*}
		Therefore with probability  at least 3/4  we have 
		$\sup_{1\leq j \leq M}|R(j)| \leq 4M^{1/2},$
		and  for such $\omega$

		\begin{equation*}
			\begin{aligned}
				\Big\|\sum_{j=1}^M e^{2\pi i yR(j)(\omega)}\Big\|_{p}^p  \geq \int_{0}^{1/200M^{1/2}}\Big|\sum_{j=1}^M \cos 2\pi  yR(j)(\omega)\Big|^pdy \geq      
				\frac{M^{p-1/2}}{200\cdot 2^p}.
			\end{aligned}
		\end{equation*}
		Therefore 
		\begin{equation*}
			\begin{aligned}
				\mathbb{E}\Big\|\sum_{j=1}^M e^{2\pi i yR(j)}\Big\|_{p}^p &\geq     
			\frac{M^{p-1/2}}{100\cdot 2^{p+2}}.
			\end{aligned}
		\end{equation*}
		This concludes \eqref{r1} and the analogue of   \eqref{s2}. Assuming $M$ to be large, and repeating this Chebyshev type argument  we see that except for a set  of probability at most
		$1/\log^2M$ we have 
		\begin{equation*}
			\begin{aligned}
				\Big\|\sum_{j=1}^M e^{2\pi i yR(j)(\omega)}\Big\|_{p}^p    \geq      
				\frac{M^{p-1/2}}{200\cdot 2^p \log M}.
			\end{aligned}
		\end{equation*}
		On the other hand from \eqref{rw-1}
		\begin{equation*}
			\mathbb{P}\Big[\sup_{|a_j|\leq 1}\Big\|\sum_{j=1}^M a_je^{2\pi i yR(j)(\omega)}\Big\|_{p}^p \geq 2M^{p-1/2}\log M \Big]\leq \frac{1}{\log M}.
		\end{equation*}
		Therefore  for a set of probability at least  $1-2/ \log M$
		\begin{equation*}
			\begin{aligned}
				100\cdot 2^{p+2}\log^2 M\Big\|\sum_{j=1}^M e^{2\pi i yR(j)(\omega)}\Big\|_{p}^p  \geq \sup_{|a_j|\leq 1} \Big\|\sum_{j=1}^M a_je^{2\pi i yR(j)(\omega)}\Big\|_{p}^p, 
			\end{aligned}
		\end{equation*}
		and this proves the analogue of \eqref{s3}.

	\end{proof}

	
	\section{Perturbations of powers I}

This section is dedicated to proving Theorem 4. This theorem is a specific case of Theorem 7. Our aim in proving this case separately is to illustrate our ideas in the simplest and most concrete case possible. Even in this case the proof is rather long and complicated. So briefly lying out our plan will be helpful for the reader.

 The results of the theorem will follow once we prove \eqref{pp1} for $p=2,4.$ The case $p=2$ has already been done for $d=1$  in the proof of  Theorem 2, and the higher $d$ cases are similar. Main issue is to prove the $p=4$ case. In this case there will be 4 variables $j_1,j_2,k_1,k_2$ instead of just 2, and this makes calculating probabilities difficult. Our strategy is  to partition our sum, and first remove the pieces that in essence do not involve 4 different variables. Then we also remove the cases like $j_1<k_1,j_2<k_2$ that can only hold with small probability. Then  we partition our sum into two cases in which   intervals $(j_1,k_1),(k_2,j_2)$ are disjoint, or intersect. The first case by independence will enable calculating probabilities, and after removing tail estimates coming from the Poisson process, we will reduce to a lattice point counting problem.  Finally there will remain the case of   intersecting intervals, and we will show that this can be reduced to the case of nonintersecting intervals by removing the intersection.
	
	\begin{proof}
		We start with  showing \eqref{pp1} for $p=2.$

		\begin{equation*}
			\begin{aligned}
				\mathbb{E}\Big\|\sum_{j\in  A} e^{2\pi i yN(j^d)}\Big\|_{2}^2  = \mathbb{E} \sum_{j,k\in A} \int_{\mathbb{T}} e^{2\pi i y[N(j^d)-N(k^d)]}dy= \sum_{j,k\in A}\mathbb{P}[N(j^d)=N(k^d)]
				= \sum_{j,k\in A}e^{-|j^d-k^d|}.
			\end{aligned}
		\end{equation*}
		We estimate this last term as
		
		\begin{equation}\label{pp4}
			\begin{aligned}
				|A|\leq \sum_{j,k\in A}e^{-|j^d-k^d|}= |A|+ 2\sum_{k\in A}\sum_{\substack {j\in A\\ j<k}}e^{j^d-k^d}&\leq |A|+ 2\sum_{k\in A}\sum_{j=1}^{(k-1)^d}e^{-k^d+j} \\
				&  \leq  |A|+ 4\sum_{j\in A}e^{-d(j-1)^{d-1}-1},
			\end{aligned}
		\end{equation}
	and this final term is bounded by $ |A| +4.$	This concludes \eqref{pp1} for $p=2.$  
		
		We now move on to  $p=4$, after which other results of the theorem will follow. The left hand side of \eqref{pp1} follows immediately from the $p=2$ case, and the Hölder inequality. So we will concentrate on the right hand side. We start with some reductions that will help us later on.  If $|A|\leq 3e^{100}$, then  our result is immediate. If $|A|>3e^{100}$ then  we let $A_0:=A\cap[1, e^{100}]$ . We obtain
		
		\begin{equation*}
			\begin{aligned}
				\mathbb{E}\Big\|\sum_{j\in  A} e^{2\pi i yN(j^d)}\Big\|_{4}^4  &\leq 	8\Big[\mathbb{E}\Big\|\sum_{j\in  A_0} e^{2\pi i yN(j^d)}\Big\|_{4}^4 +\mathbb{E}\Big\|\sum_{j\in A\setminus A_0} e^{2\pi i yN(j^d)}\Big\|_{4}^4\Big].
			\end{aligned}
		\end{equation*}
		The first term in the paranthesis can be bounded by $|A_0|^4$. If we had our   result  for sets  that contain no element in $[1, e^{100}]$, and have more elements than $e^{100}$, then   applying it  would bound the second term by $C|A\setminus A_0|^2.$ Thus the whole right hand side would be bounded by $8(C+e^{200})|A|^2.$ 
		So we may assume that  $A$ contains no element in $[1, e^{100}]$ and has more elements than $e^{100}$.  Here we observed this reduction for $d>2$, but the same reduction is possible for $d=2$ as well.

		We start with transforming our sum 		
		\begin{equation*}
			\begin{aligned}
				\mathbb{E}\Big\|\sum_{j\in  A} e^{2\pi i yN(j^d)}\Big\|_{4}^4  &= \mathbb{E} \sum_{j_1,j_2,k_1,k_2\in A} \int_{\mathbb{T}} e^{2\pi i y[N(j_1^d)+N(j_2^d)-N(k_1^d)-N(k_2^d)]}dy\\ &= \sum_{j_1,j_2,k_1,k_2\in A}\mathbb{P}[N(j_1^d)+N(j_2^d)=N(k_1^d)+N(k_2^d)].
			\end{aligned}
		\end{equation*}
	From this last sum we want to remove the cases when two of these variables are equal. These give rise to auxiliary sums that we  deal with before going to the main sum.  	The set $A^4$ to which any quadruple $(j_1,j_2,k_1,k_2)$ belongs can be written as 
		\[A^4=A_{11}\cup A_{12}\cup A_{21}\cup A_{22}\cup A_3  . \]
		Here $A_{ab}, \ a,b\in \{1,2\}$ is the set of quadruples  $(j_1,j_2,k_1,k_2)$
		for which $j_a=k_b.$ The set $A_3$ contains the remaining elements of $A^4$. 
		We have 
		\begin{equation*}
			\begin{aligned}
				\sum_{ (j_1,j_2,k_1,k_2)\in A_{11}}\mathbb{P}[N(j_1^d)+N(j_2^d)=N(k_1^d)+N(k_2^d)]&= \sum_{ j_1,j_2,k_2\in A}\mathbb{P}[N(j_2^d)=N(k_2^d)]\\ &=|A|\sum_{ j,k\in A}\mathbb{P}[N(j^d)=N(k^d)],
			\end{aligned}
		\end{equation*}
		and this last sum has been estimated above. For $A_{12},A_{21},A_{22}$ via the same argument we get the same result. Therefore it remains to handle $A_3$. We  decompose
		\[A_3=A_{31}\cup A_{32}\cup A_{33} \cup A_{34},   \]
		\begin{equation*}
			\begin{aligned}
				A_{31}&:=\{(j_1,j_2,k_1,k_2)\in A_3: j_1>k_1, j_2>k_2  \}, \quad
				A_{32}:=\{(j_1,j_2,k_1,k_2)\in A_3: j_1>k_1,  j_2<k_2  \},\\
				A_{33}&:=\{(j_1,j_2,k_1,k_2)\in A_3: j_1<k_1,  j_2>k_2  \}, \quad
				A_{34}:=\{(j_1,j_2,k_1,k_2)\in A_3: j_1<k_1,  j_2<k_2  \}.     
			\end{aligned}
		\end{equation*}
		Summing over $A_{31},A_{34}$ is easy, and we first handle these.
		\begin{equation*}
			\begin{aligned}
				\sum_{  A_{31}}\mathbb{P}[N(j_1^d)+N(j_2^d)=N(k_1^d)+N(k_2^d)]&=\sum_{  A_{31}}\mathbb{P}[N(j_1^d)-N(k_1^d)=N(k_2^d)-N(j_2^d)]. 
			\end{aligned}
		\end{equation*}
		As  $N(t)$ is increasing we must have $N(j_1^d)-N(k_1^d)\geq 0$ and $N(k_2^d)-N(j_2^d)\leq 0$. Therefore for these to be equal  they must both be zero. Hence
		\begin{equation*}
			\begin{aligned}
				\leq \sum_{  A_{31}}\mathbb{P}[N(j_1^d)-N(k_1^d)=0]\leq   |A|^2\sum_{\substack {j,k\in A\\ j<k}} \mathbb{P}[N(j^d)=N(k^d)].  
			\end{aligned}
		\end{equation*}
		We estimated this last sum above to be bounded by 2. The set $A_{34}$ is handled the same way.

		We now move on to  $A_{32}.$ This, together with its symmetric counterpart $A_{33}$, represent the most important, generic cases.  We have
		\begin{equation*}
			\begin{aligned}
				\sum_{  A_{32}}\mathbb{P}[N(j_1^d)+N(j_2^d)=N(k_1^d)+N(k_2^d)]&=\sum_{  A_{32}}\mathbb{P}[N(j_1^d)-N(k_1^d)=N(k_2^d)-N(j_2^d)]. 
			\end{aligned}
		\end{equation*}
		We decompose $A_{32}=A_{321}\cup A_{322}$, in the first of which  are contained those $(j_1,j_2,k_1,k_2)$ for which the intervals  $(k_1,j_1),(j_2,k_2)$  are disjoint, and in the second those for which they intersect. We first sum over $A_{321}$ which can be written as
		\begin{equation*}
			\begin{aligned}
				\sum_{  A_{321}} \mathbb{P}[N(j_1^d)-N(k_1^d)=N(k_2^d)-N(j_2^d)] 
				=&\sum_{  A_{321}}\sum_{a=0}^{\infty} \mathbb{P}[N(j_1^d)-N(k_1^d)=a=N(k_2^d)-N(j_2^d)]\\
				=&\sum_{a=0}^{\infty} \sum_{  A_{321}} \mathbb{P}[N(j_1^d)-N(k_1^d)=a]\Pp[N(k_2^d)-N(j_2^d)=a].
					\end{aligned}
			\end{equation*}
				This can be bounded by 
					\begin{equation*}
					\begin{aligned}
				\sum_{a=0}^{\infty}\sum_{ \substack{j_1,k_1 \in A \\j_1>k_1}}\sum_{\substack{j_2,k_2 \in A \\ k_2>j_2}}\mathbb{P}[N(j_1^d)-N(k_1^d)=a]\Pp[N(k_2^d)-N(j_2^d)=a]
				= 
				\sum_{a=0}^{\infty}\Big[\sum_{ \substack{j,k\in A \\j<k}}\mathbb{P}[N(k^d-j^d)=a]\Big]^2 . 
			\end{aligned}
		\end{equation*}
		If we could show that the inner sum is bounded by a constant $C(A)$ that may depend on $A$ but is independent of $a$, then
		\begin{equation*}
			\begin{aligned}
				\leq C(A)\sum_{a=0}^{\infty}\sum_{   \substack{j,k\in A \\j<k}}\mathbb{P}[N(k^d-j^d)=a]=C(A)\sum_{   \substack{j,k\in A \\j<k}}\sum_{a=0}^{\infty}\mathbb{P}[N(k^d-j^d)=a] \leq C(A)|A|^2. 
			\end{aligned}
		\end{equation*}	
		So our aim is to obtain such an estimate.

		We start with a crude estimate. For $a\in \N$ we have by Lemma \ref{l3}
		\begin{equation*}
			\begin{aligned}
				\sum_{ \substack{j,k\in A \\j<k}}	\mathbb{P}[N(k^d-j^d)=a] \leq|A|^2\sup_{t\geq 0}\Pp[N(t)=a] 
				\leq 	\frac{|A|^2}{\sqrt{2\pi a}}.
			\end{aligned}
		\end{equation*}
		So if $a\geq |A|^4$ the inner sum is bounded by an absolute constant.   
		This is also true for $a=0$ as we have calculated in \eqref{pp4}.    Thus we may assume $1\leq a <|A|^4.$

	 Let us define the function $f(x):=4 \sqrt{x\log x}$ for $x\geq 1$. 	We decompose the inner sum to two sums over sets $A_1,A_2$
		\begin{equation*}
			\begin{aligned}
				A_1:=\{(j,k)\in A^2: j<k, \  |a-(k^d-j^d)|\geq 2f(a) \},  \quad  A_2:=\{(j,k)\in A^2: j<k \}\setminus A_1.
			\end{aligned}
		\end{equation*}
		We note that if $a<e^{90}$ then $A_2$ is empty, and therefore we may assume that $a$ is large when summing over that set.

		We first consider the sum over $A_1,$ which is another auxiliary sum. Actually it is a tail estimate coming from the Poisson process taking every positive value with some probability, albeit small.  By Lemma \ref{lem2} we have $|a-(k^d-j^d)|\geq f(k^d-j^d)$ and by Lemma \ref{lem1}
		\begin{equation*}
			\begin{aligned}
				\mathbb{P}[N(k^d-j^d)=a] 
				\leq \mathbb{P}[|N(k^d-j^d)-(k^d-j^d)|\geq f(k^d-j^d)] 
				\leq 2(k^d-j^d)^{-4}.
			\end{aligned}
		\end{equation*}
		Therefore 
		\begin{equation*}
			\begin{aligned}
				\sum_{A_1}\mathbb{P}[N(k^d-j^d)=a]\leq 2\sum_{\substack{j,k\in A \\ j<k}} (k^d-j^d)^{-4}  \leq 2\sum_{j\in A} \sum_{k>j} k^{-4}\leq 1.
			\end{aligned}
		\end{equation*}

		We now sum over  $A_2.$ This is the main sum that will be estimated  by  counting   lattice points.
		By  Lemma \ref{l3} we have 
		\begin{equation*}
			\begin{aligned}
				\sum_{ A_2}\mathbb{P}[N(j^d-k^d)=a]  \leq  a^{-1/2}|A_2|.
			\end{aligned}
		\end{equation*}
		So our aim is to estimate the cardinality of $A_2$ which is a lattice point problem. We can rewrite 
		\begin{equation*}
			\begin{aligned}
				A_2= \{(j,k): j<k, \  (j^d+a-2f(a)))^{\frac{1}{d}} < k < (j^d+a+2f(a))^{\frac{1}{d}} \}.
			\end{aligned}
		\end{equation*}
		We observe by the mean value theorem that 
		\begin{equation*}
			\begin{aligned}
				(j^d+a+2f(a))^{\frac{1}{d}}-j\leq [a+2f(a)]\frac{j^{1-d}}{d}\leq aj^{1-d}. 
			\end{aligned}
		\end{equation*}
		When $j>a^{1/(d-1)}$ this final expression is less than 1. Therefore for these $j$ we have no pair $(j,k)$ in $A_2$. Similarly 
		\begin{equation*}
			\begin{aligned}
				(j^d+a+2f(a))^{\frac{1}{d}}- (j^d+a-2f(a))^{\frac{1}{d}} \leq  4f(a) \frac{(j^d+a-2f(a))^{\frac{1}{d}-1}}{d} &\leq 4f(a)a^{\frac{1}{d}-1 }\\  &\leq 16a^{\frac{1}{d}-\frac{1}{2}}\log a.
			\end{aligned}
		\end{equation*}
		For $d\geq 3$ this is less than one, so for each $j\leq a^{1/(d-1)}$, there can be at most 1 solution, which means $|A_2|\leq a^{1/2}.$

		For $d=2$ this method of counting solutions for each fixed $j$ is too crude, and instead we will count them for each fixed value of $b=k-j$. Since when $b$ is fixed knowing $j$ immediately gives $k$, all we need to do is to  count $j$.
		We have $k^2-j^2=b^2+2jb$, therefore elements of $A_2$ satisfy
		
		\begin{equation*}
			\begin{aligned}
				|b^2+2jb-a| < 2f(a)
				\implies  \frac{a}{2b}-\frac{b}{2}- \frac{f(a)}{b}&< j < \frac{a}{2b}-\frac{b}{2}+ \frac{f(a)}{b}.
			\end{aligned}
		\end{equation*}
		So possible number of solutions $j$ for fixed $b$ is bounded by $1+2f(a)/b$, 
		and as $j\geq 1$, we must have 
		\[b^2+2jb < a+2f(a)\implies   b^2<2a\implies b< \sqrt{2a}.\]
		Thus summing the number of solutions over $b$
		
		\begin{equation*}
			\begin{aligned}
				|A_2|\leq\sum_{b=1}^{\lfloor \sqrt{2a} \rfloor} 1+\frac{ 2f(a)}{b}\leq  \sqrt{2a}+2f(a)\log a<10\sqrt{a}\log^{3/2} a.
			\end{aligned}
		\end{equation*}
		 By our assumption $a\leq |A|^4$ the sum over $A_2$ is bounded by $80\log^{3/2} |A|$. Also taking into account the contribution of $A_1$, we see that independent of $a$, the inner sum is bounded by $1+80\log^{3/2}|A|$.
			This concludes the sum over $A_{321}.$ 
		
		We finally consider the sum over $A_{322}$. We  partition this set into four:  
		\begin{equation}\label{par2}
			\begin{aligned}
				A_{3221}&:=\{ (j_1,j_2,k_1,k_2)\in A_{322}: [k_1,j_1] \subseteq (j_2,k_2)  \}\\
				A_{3222}&:=\{(j_1,j_2,k_1,k_2)\in A_{322}: [j_2,k_2] \subseteq (k_1,j_1)  \}\\
				A_{3223}&:=\{(j_1,j_2,k_1,k_2)\in A_{322}: k_1<j_2<j_1<k_2  \}\\
				A_{3224}&:=\{(j_1,j_2,k_1,k_2)\in A_{322}: j_2<k_1< k_2<j_1  \}.    
			\end{aligned}
		\end{equation}
		First two and the last two are handled in the same way so we will only consider $ A_{3221}, A_{3223}.$ On  $ A_{3221}$ the condition $ [k_1,j_1] \subseteq (j_2,k_2)$ means
		\begin{equation*}
			\begin{aligned}
				\sum_{  A_{3221}}\mathbb{P}[N(j_1^d)+N(j_2^d)=N(k_1^d)+N(k_2^d)]&=\sum_{  A_{3221}}\mathbb{P}[N(j_2^d)-N(k_1^d)=N(k_2^d)-N(j_1^d)]. 
			\end{aligned}
		\end{equation*}
		Since $j_2<k_1,$ but $k_2>j_1$, and since $(j_2,k_1)\cap(j_1,k_2)=\emptyset$ the last sum is  
		\begin{equation*}
			\begin{aligned}
				=&\sum_{  A_{3221}}\mathbb{P}[N(j_2^d)-N(k_1^d)=0=N(k_2^d)-N(j_1^d)]\\	=	&\sum_{  A_{3221}}\mathbb{P}[N(k_1^d)-N(j_2^d)=0]\mathbb{P}[N(k_2^d)-N(j_1^d)=0]\\ 
				\leq  &\Big[ \sum_{  k>j}\mathbb{P}[N(k^d)-N(j^d)=0]  \Big]^2. 
			\end{aligned}
		\end{equation*}
		We  estimated the sum inside the square by $2$. So this case is bounded by just $4$.
		
		On $ A_{3223}$ the condition $ k_1<j_2<j_1<k_2 $ means $(k_1,j_2)\cap(j_1,k_2)=\emptyset$, and again we reduce to the case $A_{321}.$ This finishes $A_{32}$. The set $A_{33}$ is entirely symmetric to $A_{32}$, and is handled in exactly the same way. Hence  we obtain \eqref{pp1} for $p=4$. By the Hölder inequality we obtain \eqref{pp1}. Then arguments similar to those in the previous proofs yield analogues of \eqref{s2},\eqref{s3}.

	\end{proof}
	
	In this proof we  made every decomposition explicitly. Of course this will not be possible for the general case in Theorem 7, as it would result in too many cases to deal with. We will introduce new ideas to handle this situation. Also lattice point counting problem is harder in that case, and will be handled with Theorem 8.


	\section{Arithmetic progressions of larger step size}
	
	This section pertains to Theorem 5. In this theorem
	our aim is to see  how randomization affects an exponential sum with frequencies forming an arithmetic progression when the step size is large. Our Theorem 2 shows that it does not have much of an effect when the step size is $1$. But as the step size grows we will observe  the decay  induced. This  decay also means the destruction of arithmetic progression structure. It is also  instructive to compare this theorem to Theorem 4. Picking $r=d-1$ we obtain two sequences $\{j^d\},  \{jM^{r}\}, \ 1\leq j\leq M$ that grow roughly at the same pace, but for the first one growth starts slowly and then gains pace, while for the second it is uniform. Comparing the two theorems illuminates the effects that these differences between sequences  have.

	\begin{proof}
	
		We first look at the 
		$p=2$ case. 
		\begin{equation*}
			\begin{aligned}
				\mathbb{E}\Big\|\sum_{j=1}^M e^{2\pi i yN(jM^r)}\Big\|_{2}^2 = \sum_{j,k=1 }^Me^{-|j-k|M^r}= M+ 2\sum_{k=2}^Me^{-kM^{r}} \sum_{j< k }e^{jM^{r}}&\leq M+4(M-1)e^{-M^r} \\ &\leq  M+2\min\{ M, C_r\},
			\end{aligned}
		\end{equation*}
	 where $C_r$ is a constant that depends only on $r.$
		Armed with these we move on to the case $p=4$. We first show the right hand side.
		\begin{equation*}
			\begin{aligned}
				\mathbb{E}\Big\|\sum_{j=1}^M  e^{2\pi i yN(jM^r)}\Big\|_{4}^4 = \sum_{j_1,j_2,k_1,k_2=1}^M \mathbb{P}[N(j_1M^r)+N(j_2M^r)=N(k_1M^r)+N(k_2M^r)].
			\end{aligned}
		\end{equation*}
		Let $A:=\{1,2,3\ldots, M\}$, and  we decompose the set $A^4$ to which any quadruple $(j_1,j_2,k_1,k_2)$ belongs  as in the proof of Theorem 4,
		\[A^4=A_{11}\cup A_{12}\cup A_{21}\cup A_{22}\cup A_3, \]
		and all except $A_3$ are dealt with easily as before. For example on $A_{11}$
		\begin{equation*}
			\begin{aligned}
				\sum_{  A_{11}}\mathbb{P}[N(j_2M^r)=N(k_2M^r)]=M\sum_{  j,k=1}^M\mathbb{P}[N(|j-k|M^r)=0] = M\sum_{j,k=1 }^Me^{-|j-k|M^r},
			\end{aligned}
		\end{equation*}
		which as above gives a $\approx M^2$ term. Each of the $A_{ij},  \ i,j=1,2$ gives the same contribution. We thus proceed to $A_3$, and decompose
		\[A_3=A_{31}\cup A_{32}\cup A_{33} \cup A_{34}   \]
		as in the proof of Theorem 4, and 
		$A_{31},A_{34}$ can be dispatched  using exactly the same ideas. For example the sum on $A_{31}$ is bounded by 
		\begin{equation*}
			\begin{aligned}
			M^2	\sum_{  k>j}\mathbb{P}[N(kM^r)-N(jM^r)=0]\leq C_rM^2,
			\end{aligned}
		\end{equation*} 
		with $C_r$ being as above.  The set $A_{34}$  give the same contribution.

		We now move on to  $A_{32}.$ 
		We decompose $A_{32}=A_{321}\cup A_{322}$, the first of which   contains the quadruples for which the intervals  $(k_1,j_1),(j_2,k_2)$  are disjoint, and  the second  contains those for which they intersect. Mimicking the steps in section 6, the sum over $A_{321}$  is bounded by 
		\begin{equation*}
			\begin{aligned}
			\sum_{a=0}^{\infty}\Big[\sum_{  {k>j}}\mathbb{P}[N((k-j)M^r)=a]\Big]^2 =	\sum_{a\leq  M^{2r}}+\sum_{a>  M^{2r}}\Big[\sum_{  {k>j}}\mathbb{P}[N((k-j)M^r)=a]\Big]^2={\bf I + II.}
			\end{aligned}
		\end{equation*}	
	This decomposition and the methods we employ to estimate each sum  differ significantly from our previous methods.
Let $b=k-j$, and $a\in \N$. Then 
		\begin{equation*}
			\begin{aligned}
				\sum_{  {k>j}}\mathbb{P}[N((k-j)M^r)=a]=\sum_{b=1}^{M-1} (M-b)\mathbb{P}[N(bM^r)=a]&=\frac{M^{ar}}{a!}\sum_{b=1}^{M-1} (M-b)b^ae^{-bM^r}.
			\end{aligned}
		\end{equation*}
		As $b^ae^{-bM^r}$  for $b\geq 0$ is a function that increases up to a supremum and then decreases, this can be estimated by 
		\begin{equation*}
			\begin{aligned}
				\leq \frac{M^{1+ar}}{a!}\sum_{b=1}^{M-1} b^ae^{-bM^r} \leq \frac{M^{1+ar}}{a!}2\sup_{b\geq 0} b^ae^{-bM^r}+ \frac{M^{1+ar}}{a!}\int_{0}^{\infty} b^ae^{-bM^r}db. 
			\end{aligned}
		\end{equation*}
		The supremum is attained when $b=aM^{-r}$, while the integral is clearly the Gamma function after a change of variables. Plugging these in, and using \eqref{robs} gives
		\begin{equation*}
			\begin{aligned}
				=   2\frac{M a^a}{a!e^a} +
				 \frac{M^{1+ar}}{a!} \frac{\Gamma(a+1)}{M^{r(a+1)}} \leq \frac{M}{\sqrt{ a}}+M^{1-r}. 
			\end{aligned}
		\end{equation*}
		In the summation range of ${\bf I}$ this last sum is bounded by $2M/\sqrt{a}$, and thus
		\begin{equation*}
			\begin{aligned}
			{\bf I }\leq \Big[\sum_{  {k>j}}\mathbb{P}[N((k-j)M^r)=0]\Big]^2  +\sum_{1\leq a\leq  M^{2r}}\frac{4M^2}{a}\leq C_r^2+4M^2(1+2r\log M).
			\end{aligned}
		\end{equation*}	
			In the summation range of ${\bf II}$  the term  $M^{1-r}$ dominates, hence
		\begin{equation*}
			\begin{aligned}
				{\bf II }\leq  2M^{1-r} \sum_{ a>  M^{2r}}\sum_{  {k>j}}\mathbb{P}[N((k-j)M^r)=a]= 2M^{1-r} \sum_{  k>j}\sum_{ a>  M^{2r}}\mathbb{P}[N((k-j)M^r)=a]\leq M^{3-r}.
			\end{aligned}
		\end{equation*}	
	 This  finishes the estimate on $A_{321}.$

		As in the previous proof, this case of nonintersection is the generic case, and intersection cases contained in $A_{322}$ reduce to the previous cases.
		The set $A_{33}$ is similar to  $A_{32}$. We thus obtain the upper bound $\lesssim_r M^2\log M+M^{3-r}$.  
		
		For  the lower bound, we  use  the  observation that on the torus $L^4$ norm dominates $L^2$ norm	to obtain the $M^2$ term, and the argument   that in  a small neighborhood of the origin  there is little cancellation between the terms of the exponential sum to obtain the   $M^{3-r}$ term.
		
	\end{proof}
	

	\section{Perturbations of  Green-Ruzsa Sets}
	
	In this section we will prove that if we  perturb the Green-Ruzsa sets, they will satisfy the Hardy-Littlewood majorant property almost surely.  For an integer $D\geq 5$ the Green-Ruzsa set $\Lambda_{D,k}$ is defined by 
	\begin{equation*}
		\Lambda_{D,k}:= \Big\{\sum_{j=0}^{k-1}d_j D^j |  \  k\in \N,  \   d_j\in \{0,1,3\}  \Big\}.
	\end{equation*}
	Thus this set consists of nonnegative integers that are when written in base $D$ have digits only $0,1,3$. 
	Therefore this set is sparse, and we can quantify this sparsity by the  estimate below. This property is the only property of this set that will be used.
	\begin{proposition}
		The Green-Ruzsa set $\Lambda_{D,k}$ satisfies the sparsity condition
		\begin{equation*}
			\big|\Lambda_{D,k} \cap [n-M,n+M]\big|\leq 24M^{\frac{\log 3}{\log D}}
		\end{equation*}
		for any $M\in \N,  n\in \Z$. 
	\end{proposition}
	
	\begin{proof}
		We may assume $n\geq 0$.  If the set in question has less than 2 elements, we are done. Therefore we may assume that it has at least two elements, and we let $a$ be its least element and $b$ its largest element. 
		Then we can write for some $m\leq k$
		\[b-a=\sum_{j=0}^{m-1}c_jD^j \quad -3\leq c_j\leq 3\] 	
		with $c_{m-1}>0$ as $b-a>0.$ But observe that 
		\[  2M\geq  b-a \geq D^{m-1}-  3\sum_{j=0}^{m-2}D^j =D^{m-1}-3\frac{D^{m-1}-1}{D-1}\geq D^{m-1}-\frac{3}{4}D^{m-1}=\frac{D^{m-1}}{4}.\]
		From this we deduce
		\[  m\leq 1+\frac{\log 8M}{\log D}.\]
		By changing the  first $m$ digits of $a$   in base $D$ we might obtain	 elements of $\Lambda_{D,k}$ not exceeding $b$, but we cannot change the higher digits even if they exist, for numbers thus obtained will certainly lie outside $[a,b]$. Therefore there are at most $3^m$ elements of $\Lambda_{D,k}$ in $[a,b]$, which means    
		\[	\big|\Lambda_{D,k} \cap [n-M,n+M]\big| =\big|\Lambda_{D,k} \cap [a,b]\big| \leq 3^{1+\frac{\log 8M}{\log D}}=3(8M)^{\frac{\log 3}{\log D}}\leq 24 M^{\frac{\log 3}{\log D}}.   \]

	\end{proof}

	We  now prove Theorem 6.

	\begin{proof}

		We start with  \eqref{gr1} for $p=2.$
		\begin{equation*}
			\begin{aligned}
				\mathbb{E}\Big\|\sum_{j\in  A} e^{2\pi i yN(j)}\Big\|_{2}^2  = \sum_{j,k\in A}\mathbb{P}[N(|j-k|)=0]
				= \sum_{j,k\in A}e^{-|j-k|}.
			\end{aligned}
		\end{equation*}
		We estimate this last term as
		
		\begin{equation*}
			\begin{aligned}
				|A|\leq \sum_{j,k\in A}e^{-|j-k|}= |A|+ 2\sum_{k\in A}\sum_{\substack {j\in A\\ j<k}}e^{j-k}&\leq |A|+ 2\sum_{k\in A}\sum_{j=0}^{k-1}e^{j-k} \leq  3|A|.
			\end{aligned}
		\end{equation*}
		Here we note that if for example $A=\Lambda_{D,k} \cap [0,M]$, then any element $a$ of $A$ with $a\equiv 1 \pmod{D}$ there exist $b\in A$ with $b=a-1$. Also if  $a\equiv 3 \pmod{D}$, there exist  $b\in A$ with $b=a-2$. Therefore for large $|A|$ 
		\begin{equation*}
			\begin{aligned}
				\sum_{k\in A}\sum_{\substack {j\in A\\ j<k}}e^{j-k}\geq  \frac{|A|-2}{3}\big[e^{-1}+e^{-2}   \big].
			\end{aligned}
		\end{equation*}
		Thus for such sets the contribution of the nondiagonal term is comparable to that of  the diagonal term, as opposed to the situation we encountered in estimates over powers. Now all claims of the theorem for $p=2$ follow.

		We now move on to \eqref{gr1} for $p=4$. From this \eqref{gr1} follows for all $2\leq p\leq 4$ immediately.  The left hand side of \eqref{gr1} follows immediately from the $p=2$ case, and the Hölder inequality. So we  concentrate on the right hand side. We may assume that $A$ contains no element in $[0, e^{100}]$, have more elements than $e^{100}$, and whenever $j,k\in A, \ k>j$ we have $k-j\geq e^{100}$. For if we have our result under these assumptions, then for a  generic set $A\subset \Z_+$ we can do the following decomposition 
		\begin{equation*}
			A_{-1}:=A\cap [0,e^{100}],   \qquad  	A_n=(A\setminus A_{-1})\cap(\lceil e^{100} \rceil \Z +n), \qquad   0\leq n <\lceil e^{100} \rceil.
		\end{equation*}
		Some $A_n,  \ 0\leq n <  \lceil e^{100} \rceil  $ may have more elements than $e^{100}$, let $N_1$ be the set of these $n$, and let $N_2$ be the set of the other $n$, including $n=-1$. Then 
		\begin{equation*}
			\begin{aligned}
				\mathbb{E}\Big\|\sum_{j\in  A} e^{2\pi i yN(j)}\Big\|_{4}^4  &\leq e^{303}\Big[\sum_{n\in N_1} \mathbb{E}\Big\|\sum_{j\in  A_n} e^{2\pi i yN(j)}\Big\|_{4}^4  + \sum_{n\in N_2} \mathbb{E}\Big\|\sum_{j\in  A_n} e^{2\pi i yN(j)}\Big\|_{4}^4    \Big].
			\end{aligned}
		\end{equation*}
		Applying our result under the assumptions above to $A_n ,\ n\in N_1$, and estimating sums over $A_n ,\ n\in N_2$ trivially 	
		\begin{equation*}
			\begin{aligned}
				\leq e^{303}\Big[C\sum_{n\in N_1}|A_n|^2  + \sum_{n\in N_2} |A_n|^4  \Big] 	
				\leq e^{303}\Big[C\sum_{n\in N_1}|A_n|^2  + e^{202}\sum_{n\in N_2} |A_n|^2 \Big] \leq e^{303}[C+e^{202}]|A|^2.
			\end{aligned}
		\end{equation*}
		So proving our result under these assumptions suffices.

		We start with
		\begin{equation*}
			\begin{aligned}
				\mathbb{E}\Big\|\sum_{j\in  A} e^{2\pi i yN(j)}\Big\|_{4}^4  = \sum_{j_1,j_2,k_1,k_2\in A}\mathbb{P}[N(j_1)+N(j_2)=N(k_1)+N(k_2)].
			\end{aligned}
		\end{equation*}
		As before, the set $A^4$ can be written as 
		$A^4=A_{11}\cup A_{12}\cup A_{21}\cup A_{22}\cup A_3, $ 
		and for any $A_{ij},i,j=1,2$ 
		\begin{equation*}
			\begin{aligned}
				\sum_{  A_{ij}}\mathbb{P}[N(j_1^d)+N(j_2^d)=N(k_1^d)+N(k_2^d)]=|A|\sum_{ j,k\in A}\mathbb{P}[N(j^d)=N(k^d)].
			\end{aligned}
		\end{equation*}
		This last sum has been estimated above. So from these sets we have $\lesssim |A|^2$ contribution.   
		
		We  now decompose
		$A_3=A_{31}\cup A_{32}\cup A_{33} \cup A_{34} $ as before.
		Because we lack the good bounds on the nondiagonal sum in the $p=2$ case, unlike the powers case, estimating the  sets $A_{31},A_{34}$ is now more difficult.  Without loss of generality we can just concentrate on $A_{31}.$
		\begin{equation*}
			\begin{aligned}
				\sum_{  A_{31}}\mathbb{P}[N(j_1)+N(j_2)=N(k_1)+N(k_2)]&=\sum_{  A_{31}}\mathbb{P}[N(j_1)-N(k_1)=0=N(k_2)-N(j_2)]. 
			\end{aligned}
		\end{equation*}
	At this stage if we just ignore one of these equations and crudely estimate by 
		\begin{equation*}
			\begin{aligned}
				\leq  |A|^2\sum_{  \substack{j,k\in A \\ j<k}}\mathbb{P}[N(k)-N(j)=0], 
			\end{aligned}
		\end{equation*}
		we may get a contribution $\approx |A|^3$. Therefore we cannot follow this approach, and must employ more  delicate analysis similar in vein to the one we deploy for the sets $A_{32},A_{33}.$  We partition  
		$	A_{31}$ into two  sets $A_{311},A_{312}$ where the first contains the quadruples  for which $ (k_1,j_1),(k_2,j_2)$ are disjoint, and the second those for which they intersect.
		For $A_{311}$ we have

		\begin{equation*}
			\begin{aligned}
				\sum_{  A_{311}}\mathbb{P}[N(j_1)-N(k_1)=0=N(j_2)-N(k_2)]
				=& \sum_{  A_{311}}\mathbb{P}[N(j_1)-N(k_1)=0]\Pp[ N(j_2)-N(k_2)=0]
				\\ \leq &\Big[\sum_{ \substack{j,k\in A \\ j<k}}\mathbb{P}[N(k)-N(j)=0] \Big]^2 .  
			\end{aligned}
		\end{equation*}
		This last sum, as we estimated above is  $\lesssim |A|^2$. In $A_{312}$ we have two possibilities, either
		$j_2-k_2\geq j_1-k_1$ or $j_2-k_2< j_1-k_1$. Without loss of generality we may assume the first. Then 
		\begin{equation*}
			\begin{aligned}
				\sum_{ \substack{ A_{312}\\ j_2-k_2\geq j_1-k_1  } }\mathbb{P}[N(j_1)-N(k_1)=0=N(j_2)-N(k_2)]
				 \leq \sum_{   \substack{ A_{312}\\ j_2-k_2\geq j_1-k_1  }}\Pp[ N(j_2)-N(k_2)=0].  
			\end{aligned}
		\end{equation*}	
		The conditions that $(k_1,j_1)$ must intersect $(k_2,j_2)$, and $j_2-k_2\geq j_1-k_1$ forces $j_1,k_1$ to be within an interval of length $3(j_2-k_2)-2$. Considering the sparsity condition on $A$, there are 
		$		C_A [ (3(j_2-k_2)/2]^{\alpha}$
	integers to choose from.  Thus the last sum above is bounded by 
		\begin{equation*}
			\begin{aligned}
				\leq 2C^2_A\sum_{\substack{j,k \in A \\  k>j}} \Pp[ N(k-j)=0](k-j)^{2\alpha}\leq   2C^2_A\sum_{\substack{j,k \in A \\  k>j}} e^{j-k}(k-j)^{2\alpha} \leq  2C^2_A |A|^2\sup_{x\geq 0} xe^{-x}  \leq  2C^2_A|A|^2. 
			\end{aligned}
		\end{equation*}
		Thus contribution from $A_{31},A_{34}$ is $\lesssim_{C_A} |A|^2.$

		We now move on to  $A_{32}.$ This, together with its symmetric counterpart $A_{33}$, which is handled in exactly the same way, represent the most important, generic cases.  We decompose $A_{32}=A_{321}\cup A_{322}$ as before. 
	We can write 	the sum over $A_{321}$  as before
		\begin{equation*}
			\begin{aligned}
				\sum_{  A_{321}} \mathbb{P}[N(j_1)+N(j_2)=N(k_1)+N(k_2)] 
				\leq 
				\sum_{a=0}^{\infty}\Big[\sum_{ \substack{j,k\in A \\j<k}}\mathbb{P}[N(k-j)=a]\Big]^2, 
			\end{aligned}
		\end{equation*}
		and use the ideas introduced above, but this  is not good enough. For the extra variable $a$ introduced to relate $j_1,j_2,k_1,
		k_2$ to each other leads to inefficiencies. Instead we will relate these variables to each other directly with an inequality.  We further decompose $A_{321}$ into the set of quadruples $A_{3211}$ with $k_2-j_2\geq j_1-k_1$, which we consider without loss of generality, and the remaining ones comprising $A_{3212}$ . We define $f(x):=4\sqrt{x\log x}$ for $x\geq 1.$
		To each quadruple $(j_1,j_2,k_1,k_2)$ we consider,  we assign three events
		\begin{equation*} 
			\begin{aligned}
				& \Omega_*:=\big\{\omega\in \Omega: N(j_1)(\omega)-N(k_1)(\omega)=N(k_2)(\omega)-N(j_2)(\omega)  \big\},\\
				&\Omega_1:=\big\{\omega\in \Omega: |N(j_1)(\omega)-N(k_1)(\omega)-(j_1-k_1)|< f(k_2-j_2)  \big\}, \\
				&\Omega_2:=\big\{\omega\in \Omega: |N(k_2)(\omega)-N(j_2)(\omega)-(k_2-j_2)|< f(k_2-j_2)  \big\}. 
			\end{aligned}
		\end{equation*}
		Using these sets we can write
		\begin{equation*}
			\begin{aligned}
				\Omega_*\subseteq [\Omega_*\cap\Omega_1 \cap \Omega_2] \cup \Omega_1^c \cup \Omega_2^c,
			\end{aligned}
		\end{equation*}
		and thus also
		\begin{equation*}
			\begin{aligned}
				\Pp[\Omega_*]\leq \Pp[\Omega_*\cap\Omega_1 \cap \Omega_2] + \Pp[\Omega_1^c] +\Pp[\Omega_2^c].
			\end{aligned}
		\end{equation*}
		We  have by Lemma \ref{lem1}
		\begin{equation*}
			\begin{aligned}
				\Pp[\Omega_1^c], \Pp[\Omega_2^c]\leq  2e^{-4\log (k_2-j_2)}=2(k_2-j_2)^{-4}, 
			\end{aligned}
		\end{equation*}
		that is, the contribution of these terms are minor. For the main term  $\Omega_*\cap\Omega_1 \cap \Omega_2$ we first observe that if this set  contains even one   $\omega$, we have by triangle inequality
		\begin{equation*}\label{gr4} 
			\begin{aligned}
				&|(k_2-j_2)-(j_1-k_1)|\\ =&\big|(k_2-j_2)-(j_1-k_1)+N(j_1)(\omega)-N(k_1)(\omega)-[N(k_2)(\omega)-N(j_2)(\omega)] \big| \\  \leq &|N(j_1)(\omega)-N(k_1)(\omega)-(j_1-k_1)|+ |N(k_2)(\omega)-N(j_2)(\omega)-(k_2-j_2)|\\   < &2f(k_2-j_2). 
			\end{aligned}
		\end{equation*}
		So the set $\Omega_*\cap\Omega_1 \cap \Omega_2$ is empty  for quadruples violating this relation. On the other hand, the probability of $\Omega_*$ can be estimated as
		\begin{equation*} 
			\begin{aligned}
				\Pp[\Omega_*]=\sum_{a=0}^{\infty} \Pp[ N(j_1)-N(k_1)=a=N(k_2)-N(j_2) ]  & =\sum_{a=0}^{\infty} \Pp[ N(j_1-k_1)=a]\Pp[N(k_2-j_2)=a ]\\ &	\leq \frac{1}{\sqrt{k_2-j_2}}\sum_{a=0}^{\infty} \Pp[ N(j_1-k_1)=a]\\ &	=\frac{1}{\sqrt{k_2-j_2}}.
			\end{aligned}
		\end{equation*}
		We then can bound
		\begin{equation*}
			\begin{aligned}
				\sum_{A_{3211}} \mathbb{P}[N(j_1)+N(j_2)=N(k_1)+N(k_2)] =
				\sum_{  A_{3211}} \mathbb{P}[N(j_1)-N(k_1)=N(k_2)-N(j_2)] 
			\end{aligned}
		\end{equation*}
	by the following			
			\begin{equation*}
				\begin{aligned}
				 	&\sum_{  A_{3211}} \Pp[\Omega_*\cap\Omega_1 \cap \Omega_2] + \Pp[\Omega_1^c] +\Pp[\Omega_2^c]
				\leq \sum_{  \substack{A_{3211}\\  |(k_2-j_2)-(j_1-k_1)|< 2f(k_2-j_2)  }}  \frac{1}{\sqrt{k_2-j_2}}+ \sum_{  A_{3211}}   4(k_2-j_2)^{-4}.  
			\end{aligned}
		\end{equation*}
		The second sum is easy to handle:
		\begin{equation*}
			\begin{aligned}
				\sum_{  A_{3211}}   4(k_2-j_2)^{-4}\leq  \sum_{\substack{j_2,k_2 \in A \\ j_2  <k_2 }}\sum_{\substack{j_1,k_1 \in A \\ k_1<j_1}}   4(k_2-j_2)^{-2}(j_1-k_1)^{-2} \leq  4\Big[\sum_{\substack{j,k \in A \\ j<k}}   (k-j)^{-2}\Big]^2.  
			\end{aligned}
		\end{equation*}
		But we have
		\begin{equation*}
			\begin{aligned}
				\sum_{\substack{j,k \in A \\ k>j}}     (k-j)^{-2} \leq \sum _{k\in A}\sum_{1<j<k}   (k-j)^{-2} < \sum _{k\in A}\frac{\pi ^2}{6}<2|A|.  
			\end{aligned}
		\end{equation*}
		So the contribution of the second sum is bounded by $16|A|^2.$ 
		
		We now look at the first sum. It is bounded by 
		\begin{equation*}
			\begin{aligned}
				\sum_{\substack{j,k \in A \\ j_2<k_2}} \frac{1}{\sqrt{k_2-j_2}}|\{(j_1,k_1)\in A^2:  |(k_2-j_2)-(j_1-k_1)|< 2f(k_2-j_2)   \}|.
			\end{aligned}
		\end{equation*}
		So we need to bound the cardinality of the set within the summation. The condition $|(k_2-j_2)-(j_1-k_1)|< 2f(k_2-j_2)$ implies that for a fixed $k_1 \in A$  the element $j_1$ lies  in an interval of radius $2f(k_2-j_2)$.  By the density condition on our set $A$, there are at most $C_A(2f(k_2-j_2))^{\alpha}$ elements $j_1$ for each $k_1 \in A$. Therefore 
		\begin{equation*}
			\begin{aligned}
				\leq 2C_A	|A|	\sum_{\substack{j,k \in A \\ j<k}} \sqrt{\log(k-j)} (k-j)^{\frac{\alpha-1}{2}}\leq 	2C_A	|A|	\sum_{n=0}^{\infty} \sum_{\substack{j,k \in A \\ 2^n\leq k-j<2^{n+1}}} ((n+1)\log 2)^{\frac{1}{2}} 2^{n\frac{\alpha-1}{2}}.
			\end{aligned}
		\end{equation*}
		To proceed we need the cardinality of $\{(j,k)\in A^2: 2^n\leq k-j<2^{n+1}\}$. For a pair $(j,k)$ to be in this set, for a fixed $j$ we must have $k$ within the interval $[j+2^{n},j+2^{n+1})$. By the density hypothesis on $A$, this means for a fixed $j$ we can have at most $C_A2^{n\alpha}$ values of $k.$
		Thus
		\begin{equation*}
			\begin{aligned}
				\leq 	2C_A^2	|A|^2	\sum_{n=0}^{\infty}  (n+1)^{\frac{1}{2}} 2^{n\frac{3\alpha-1}{2}}.
			\end{aligned}
		\end{equation*}
		Plainly if $\alpha<1/3$ the sum converges to a constant $C_{\alpha}.$ This concludes the estimation of the sum $A_{321}$ with a constant that depends only on $\alpha,C_A.$

		We finally consider the sum over $A_{322}$. We  partition this set into $A_{322n}, \ 1\leq n\leq 4$ as in   
	\eqref{par2}.
	The first two and the last two are handled in the same way so we will only consider $ A_{3221}, A_{3223}.$ 
	For  $ A_{3221}$  the same arguments as in Chapter 6 yield
	
		\begin{equation*}
			\begin{aligned}
				\sum_{  A_{3221}}\mathbb{P}[N(j_1)+N(j_2)=N(k_1)+N(k_2)]\leq  \Big[ \sum_{  k>j}\mathbb{P}[N(k-j)=0]  \Big]^2. 
			\end{aligned}
		\end{equation*}
		We  estimated the sum inside the square by $3|A|$. So the contribution from here is at most $9|A|^2.$
		
		On $ A_{3223}$ the condition $ k_1<j_2<j_1<k_2 $ means $(k_1,j_2)\cap(j_1,k_2)=\emptyset$, and again we reduce to the case $A_{321}.$ Hence  we obtain \eqref{gr1} for $p=4$. By the Hölder inequality we obtain \eqref{gr1}.

		Plainly for any $n\in \N$
		\begin{equation*}
			\begin{aligned}
				\Big\|\sum_{j\in  A} a_je^{2\pi i yN(j)}\Big\|_{2n}^{2n} \leq  \Big\|\sum_{j\in  A} e^{2\pi i yN(j)}\Big\|_{2n}^{2n}.
			\end{aligned}
		\end{equation*}
		 Then \eqref{gr2} follows by using  this together with  \eqref{gr1}, and the Hölder inequality.
		
		For $p\geq 2$
		\begin{equation*}
			\begin{aligned}
				\Big\|\sum_{j\in  A} e^{2\pi i yN(j)}\Big\|_{p}^{p} \geq	\Big\|\sum_{j\in  A} e^{2\pi i yN(j)}\Big\|_{2}^{p} \geq  |A|^{p/2}.
			\end{aligned}
		\end{equation*}
		So by the Chebyshev inequality 
		\begin{equation*}\label{}
			\begin{aligned}
				\mathbb{P}\Big[\sup_{|a_j|\leq 1} \Big\|  \sum_{j\in A} a_je^{2\pi i yN(j)}  \Big\|_{p}^p \geq |A|^{\varepsilon}\Big\|  \sum_{j\in A} e^{2\pi i yN(j)}   \Big\|_{p}^p \Big]   
				&\leq  \mathbb{P}\Big[\sup_{|a_j|\leq 1} \Big\|  \sum_{j\in A} a_je^{2\pi i yN(j)}  \Big\|_{p}^p \geq |A|^{\frac{p}{2}+\varepsilon} \Big]\\
			&	\leq |A|^{-\frac{p}{2}-\varepsilon}	\mathbb{E}\sup_{|a_j|\leq 1}  \Big\|  \sum_{j\in A} a_je^{2\pi i yN(j)}\Big\|_p^p .
			\end{aligned}
		\end{equation*}
		By  \eqref{gr2} this last expression is  $ \lesssim_{\alpha,C_A}|A|^{-\varepsilon}.$ Therefore taking the limits completes the proof of \eqref{gr3}.

	\end{proof}


	\section{Perturbations of Powers II:  $p=2n$}

	In this section we prove Theorem 7. The proof handles the explicit decompositions of the proofs of Theorems 4,5,6 via abstract combinatorial arguments.  Also in those proofs when two variables are equal we reduce to the  $p=2$ case. In the present proof, this  reduces us to $p=2n-2$, which requires an inductive argument. After finishing off the auxiliary sums, the main sum reduces to a lattice point problem, for which we appeal to Theorem 8.

	\begin{proof}

	 With Theorem 4 we already proved the cases $p=2,4$, so we assume $p\geq 6$, and thus $d\geq n\geq 3$. As explained we will need an inductive argument, so we also assume the statement to be true for $n-1.$
		We may assume that $A$ contains no element in $[1, e^{100d}]$, and have more elements than $e^{100d}$. It follows immediately that  whenever $j,k\in A, \ k>j$ we have $k^d-j^d\geq e^{100d(d-1)}$.

		We let ${\bf j}=(j_1,j_2,\cdots,j_n)$ and ${\bf k}=(k_1,k_2,\cdots,k_n)$ denote vectors in $A^n$. We start with transforming our sum 		
		\begin{equation*}
			\begin{aligned}
				\mathbb{E}\Big\|\sum_{j\in  A} e^{2\pi i yN(j^d)}\Big\|_{2n}^{2n}  = \sum_{({\bf j,k})\in A^{2n}}\mathbb{P}\Big[\sum_{i=1}^nN(j_i^d)=\sum_{i=1}^nN(k_i^d)\Big].
			\end{aligned}
		\end{equation*}
		
		We decompose  $A^{2n}=A^*\cup A^{**}$, where $A^{**}$ contains the   vectors $({\bf j,k})$ with $j_a=k_b$ for some  $1\leq a,b\leq n.$ The main contribution comes from  $A^*$, and $A^{**}$ can easily be dealt with induction. We first dispatch $A^{**}$. 
		
		Let $A^{**}_{ab}$ for some fixed $1\leq a,b\leq n$ denote the subset of  $A^{**}$ contsisting  of vectors $({\bf j,k})$ with $j_a=k_b$.  Then after relabeling 
		\begin{equation*}
			\begin{aligned}
				\sum_{  A_{ab}^{**}}\mathbb{P}\Big[\sum_{i=1}^nN(j_i^d)=\sum_{i=1}^nN(k_i^d)\Big]&=|A|\sum_{ ({\bf j',k'})\in A^{2n-2}}\mathbb{P}\Big[\sum_{i=1}^{n-1}N(j_i^d)=\sum_{i=1}^{n-1}N(k_i^d)\Big]\\ &=
				|A|\mathbb{E}\Big\|\sum_{j\in  A} e^{2\pi i yN(j^d)}\Big\|_{2(n-1)}^{2(n-1)}. 
			\end{aligned}
		\end{equation*}
		Now applying the inductive hypothesis gives an appropriate bound for  this term.

		Therefore it remains to handle $A^*$. 
		We consider the set
		\begin{equation*}
			\begin{aligned}
				A_{1}^*&=:\{({\bf j,k})\in A^{2n}: j_1\leq j_2 \leq \ldots j_n, \ \   k_1\leq k_2\leq \ldots k_n  \}.
			\end{aligned}
		\end{equation*}
		Since there are $n!$ different orderings of  for  each of $j_i,k_i $ and since each ordering can be turned into  another by  a relabeling, we have
		\begin{equation*}
			\begin{aligned}
				\sum_{  A^*}\mathbb{P}\Big[\sum_{i=1}^n(j_i^d)=\sum_{i=1}^nN(k_i^d)\Big]&\leq (n!)^2 \sum_{A_1^*}\mathbb{P}\Big[\sum_{i=1}^nN(j_i^d)=\sum_{i=1}^nN(k_i^d)\Big].
			\end{aligned}
		\end{equation*}
		
		We will say  that the pair $j_i,k_i$ have positive orientation if $j_i<k_i$ and negative orientation if $j_i>k_i.$  For every vector $({\bf j,k})\in A^*_1$ this divides pairs or indices $1\leq i\leq n$  of the vector into two sets $\sigma_+({\bf j,k})$ and $\sigma_-({\bf j,k}).$  We observe that open intervals $(j_a,k_a)\in \sigma_+({\bf j,k})$ and $(k_b,j_b)\in \sigma_-({\bf j,k})$ cannot intersect. For intersection implies $j_a<j_b$ and hence $a<b$, but it also implies  $k_b<k_a$ and thus $b<a.$  The sum over $A_1^*$  above then becomes
		
		\begin{equation}\label{p2k1}
			\begin{aligned}
			 = \sum_{A_1^*}\mathbb{P}\Big[\sum_{i\in \sigma_-({\bf j,k})}N(j_i^d)-N(k_i^d)=\sum_{i\in \sigma_+({\bf j,k})}N(k_i^d)-N(j_i^d)\Big].
			\end{aligned}
		\end{equation}
		Of course one of $\sigma_+({\bf j,k}),\sigma_-({\bf j,k})$ may be empty, in which case the sum over that set is taken to be zero. We define $f(x):=4\sqrt{nx\log x}$ for $x\geq 1.$
		To each vector  in $A_1^*$   we assign the events
		\begin{equation*} 
			\begin{aligned}
				& \Omega_*:=\Big\{\sum_{i\in \sigma_-({\bf j,k})}N(j_i^d)-N(k_i^d)=\sum_{i\in \sigma_+({\bf j,k})}N(k_i^d)-N(j_i^d)  \Big\},\\
				&\Omega_i:=\Big\{ \big|N(k_i^d)(\omega)-N(j_i^d)(\omega)-(k_i^d-j_i^d)\big|<\max_{1\leq i\leq n} f(|k_i^d-j_i^d|)  \Big\}, \qquad 1\leq i \leq n.
			\end{aligned}
		\end{equation*}
		We let $m$ be an index  that maximizes  $|k_i^d-j_i^d|$ and hence $f(|k_i^d-j_i^d|)$.    Using these sets we  define $\Omega_{**}:=\Omega_*\cap\Omega_1 \cap \Omega_2\ldots \cap \Omega_n$, and obtain
		\begin{equation*}
			\begin{aligned}
				\Pp[\Omega_*]\leq \Pp[\Omega_{**}] + \sum_{i=1}^n\Pp[\Omega_i^c] .
			\end{aligned}
		\end{equation*}
		By Lemma \ref{lem1}, 
		$\Pp[\Omega_i^c] \leq 2|k_m^d-j_m^d|^{-4n},$ 
		that is, the contribution of these terms are minor. For the main term  $\Omega_{**}$ we first observe that if this set  is nonempty for a vector $({\bf j,k})$ we have by the triangle inequality
		\begin{equation}\label{p2k2} 
			\begin{aligned}
				\Big| \sum_{i\in \sigma_+({\bf j,k})} (k_i^d-j_i^d) -  \sum_{i\in \sigma_-({\bf j,k})} (j_i^d-k_i^d)  \Big|  <  nf(|k_m^d-j_m^d|). 
			\end{aligned}
		\end{equation}	
		So the set  $\Omega_{**}$ is empty  for vectors violating this relation.
		Let $A_{11}^*$  be the set of vectors in  $A_{1}^*$ satisfying this relation.  Observe that if for a vector $({\bf j,k})$ one of the sets  $\sigma_+({\bf j,k}),  \sigma_-({\bf j,k})$ is empty, then that vector cannot belong to $A_{11}^*$. For vectors in $A_{11}^*$, assuming $m\in \sigma_+({\bf j,k})$, we compute the probability of the set $\Omega_{**}$ as follows.  Our observation on nonintersection of pairs in sets $\sigma_-,\sigma_+$, and the independent increment property of  Poisson processes allow us to apply Lemma 5 to write
		\begin{equation*} 
			\begin{aligned}
				\Pp[\Omega_{**}]	&=\sum_{a\in \Z_+}\mathbb{P}\Big[\sum_{i\in \sigma_-({\bf j,k})}N(j_i^d)-N(k_i^d)=a=\sum_{i\in \sigma_+({\bf j,k})}N(k_i^d)-N(j_i^d)\Big] \\ 	&=\sum_{a\in \Z_+}\mathbb{P}\Big[\sum_{i\in \sigma_-({\bf j,k})}N(j_i^d)-N(k_i^d)=a\Big]\Pp\Big[ \sum_{i\in \sigma_+({\bf j,k})}N(k_i^d)-N(j_i^d)=a\Big]. 
			\end{aligned}
		\end{equation*}	
		We apply Lemma 7 to estimate the sum on $\sigma_+$, and sum over $a$ for the sum on $\sigma_-$
		\begin{equation*} 
			\begin{aligned}
				\leq \frac{\sqrt{n}}{\sqrt{ |k_m^d-j_m^d|}}\sum_{a\in \Z_+}\mathbb{P}\Big[\sum_{i\in \sigma_-({\bf j,k})}N(j_i^d)-N(k_i^d)=a\Big] \leq \frac{\sqrt{n}}{\sqrt{ |k_m^d-j_m^d|}} .
			\end{aligned}
		\end{equation*}
		If $m$ belongs to $ \sigma_-({\bf j,k})$, after using independence  we apply Lemma 7 to the sum on $\sigma_-$ and sum over $a$ the other sum to get the same result.
		Combining all of these we  continue from \eqref{p2k1}
		\begin{equation*}
			\begin{aligned}
				\leq	\sum_{  A_{1}^*}\Big[ \Pp[\Omega_{**}]+ \sum_{i=1}^n \Pp[\Omega_i^c]   \Big] 
				&	=	\sum_{  A_{11}^*} \Pp[\Omega_{**}] + \sum_{  A_{1}^*} \sum_{i=1}^n \Pp[\Omega_i^c] \\ &\leq 
				\sum_{  A_{11}^*} \frac{\sqrt{n}}{\sqrt{ |k_m^d-j_m^d|}}+ 2n\sum_{  A_{1}^*} |k_m^d-j_m^d|^{-4n}. 
			\end{aligned}
		\end{equation*}
		The second sum is easy to handle:
		\begin{equation*}
			\begin{aligned}
				\sum_{ A_{1}^*}   |k_m^d-j_m^d|^{-4n}\leq  	\sum_{ A_{1}^*}   \prod_{i=1}^{n}|k_i^d-j_i^d|^{-4} \leq 2^n \prod_{i=1}^{n}   \sum_{\substack{j_i,k_i \in A \\ j_i<k_i}}   |k_i^d-j_i^d|^{-4}  \leq  2^n\Big[\sum_{k \in A }   k^{-4d+5}\Big]^n\leq  1 .  
			\end{aligned}
		\end{equation*}

		We now look at the first sum. We decompose $ A_{11}^*=\bigcup_{i=1}^n  A_{11i}^*$ where $m=i$ for vectors in $A_{11i}^*.$   We further decompose   $ A_{11i}^*=\bigcup_{\substack{j_i,k_i\in A \\ j_i\neq k_i}}  A_{11i}^{*j_ik_i}$, where $A_{11i}^{*j_ik_i}$ contains the  vectors in    $A_{11i}^{*}$ in  which $j_i,k_i$ are fixed. Then 
		\begin{equation*}\label{}
			\begin{aligned}
				\sum_{  A_{11}^*} \frac{1}{\sqrt{|k_m^d-j_m^d|}}&\leq \sum_{i=1}^n \sum_{  A_{11i}^{*}} \frac{1}{\sqrt{ |k_i^d-j_i^d|}} \leq n\sup_{1\leq i\leq n} \sum_{  A_{11i}^{*}} \frac{1}{\sqrt{ |k_i^d-j_i^d|}}.
			\end{aligned}
		\end{equation*}
		We will estimate the last sum for a fixed $i=r.$ As the estimation will not depend on the choice of  $r$,  this will suffice.
		\begin{equation}\label{p2k4}
			\begin{aligned}	
				\sum_{  A_{11r}^{*}} \frac{1}{\sqrt{ |k_r^d-j_r^d|}}=  \sum_{\substack{j_r,k_r \in A \\ j_r\neq k_r}}  \sum_{  A_{11r}^{*j_rk_r}} \frac{1}{\sqrt{|k_r^d-j_r^d|}}=  \sum_{\substack{j_r,k_r \in A \\ j_r\neq k_r}} \frac{\big|  A_{11r}^{*j_rk_r} \big|}{\sqrt{ |k_r^d-j_r^d|}}={\bf S}.
			\end{aligned}
		\end{equation}

	We  need to estimate $|A_{11r}^{*j_rk_r}|$. 	Let $D_i:=k_i^d-j_i^d$.
	 The set $A_{11r}^{*j_rk_r}$ lies within the set 
	\begin{equation}\label{p2k5}
		\begin{aligned}
			\Big\{ ({\bf j,k})\in A^{2n}\Big| \  j_r,k_r  \  \text{are fixed}, \quad  0<|k_i^d-j_i^d|\leq |D_r| \  \  \ 1\leq i\leq n,  \quad \eqref{p2k2} \ \text{holds}    \Big\}.
		\end{aligned}
	\end{equation}
We estimate the size of this set by looking at the possible choices for pairs $j_i,k_i$. Number of choices for  any pair $j_i,k_i$ for vectors of this set can be estimated  trivially by $|A|^2$ using merely the fact that $j_i,k_i\in A$. This we call the first method.  Alternatively our second method proceeds  via the identity  
\begin{equation*}\label{}
	\begin{aligned}	
	 \big|\big\{(j_i,k_i)\in A^2 \big| \   0<|k_i^d-j_i^d|\leq |D_r|  \    \big\}   \big|=2 \big|\big\{(j_i,k_i)\in A^2 \big| \  0<k_i^d-j_i^d\leq |D_r|   \  \big\}   \big|
	\end{aligned}
\end{equation*}
and  applying Theorem 8 with  $E=D=|D_r|$, and obtains  a bound $C_d|D_r|^{2/d}.$  Once $n-1$ pairs are fixed, for the  remaining pair we can use the full force of our arithmetic results. This will be called the refined method.

 We first   crudely  estimate
 {\bf S}. Observe that by the second method
\begin{equation*}\label{}
		{\bf  S}\lesssim_{d,n} \sum_{\substack{ j_r,k_r \in A \\  j_r\neq k_r}} \frac{|D_r|^{2(n-1)/d}}{ |D_r|^{1/2}}=  \sum_{\substack{ j_r,k_r \in A \\  j_r\neq k_r}}  |D_r|^{\frac{2(n-1)}{d}-\frac{1}{2}}.
\end{equation*}
The exponent of $|D_r|$ in this last expression is  nonpositive if $d\geq 4n-4$. In this case the contribution of this term can be controlled  by $C_{d,n}|A|^2.$ So we only need to  estimate ${\bf S}$ under the assumption $d<4n-4.$

To do this we will estimate the size of \eqref{p2k5} combining all our three methods. We have $n-1$ pairs to estimate in this set. One of them will be estimated by the refined method after all $n-2$ pairs are estimated and fixed by the first and the second methods, we fix   an index  $1\leq l\leq n$ with $l\neq r$ for this purpose.  For the other $n-2$ indices we note that in general the first method gives a better bound than the second, but it is the second method that enables cancellation with the $\sqrt{|k^d-j^d|}$ term of the denominator. So we must use the second method as many times  as possible to take advantage of this cancellation, and for the remaining pairs we use the first method. To this end we pick $d':= \lceil \frac{d}{4}-\frac{1}{2} \rceil-1$ indices apart from $r,l$. This $d'$ is the largest number of applications for the second method after which the power of  
$|k^d-j^d|$ still remains nonpositive in the sum. Given the assumption $d<4n-4$ we also have $d'\leq n-2$. To the remaining $n-2-d'$ pairs we apply the first method. So the cardinality of \eqref{p2k5} is bounded by

		\begin{equation}\label{p2k3}
				\lesssim_{d,n}|D_r|^{\frac{2d'}{d}}|A|^{2(n-2-d')} \sup_{\substack{|D_i| \leq |D_r| \\ i\neq r,l}}	\Big|  \Big\{ ({ j_l,k_l})\in A^{2}\Big| \  j_l\neq k_l,     \quad \Big| k_l^d-j_l^d +\sum_{i\neq l} D_i   \Big|  <  nf(|D_r|)    \Big\}\Big|.
				\end{equation}
	
		So only  the index $l$ remains, to which we will apply the refined method. Since $\Big|\sum_{i\neq l} D_i\Big| \leq n|D_r|$ we can bound the supremum above by 
		\begin{equation*}
			\begin{aligned}
				&\leq  \sup_{|D| \leq n|D_r| }	\big|  \big\{ ({ j_l,k_l})\in A^{2}\big| \  j_l\neq k_l,     \quad  | k_l^d-j_l^d +D | < nf(|D_r|)     \big\}\big| \\ 	&\leq 2 \sup_{|D| \leq n|D_r| }	\big|  \big\{ ({ j,k})\in A^{2}\big| \  j< k,     \quad  | k^d-j^d -D | < nf(|D_r|)     \big\}\big|  \\ 	&\leq 2 \sup_{nf(|D_r|)  \leq D \leq n|D_r| }	\big|  \big\{ ({ j,k})\in A^{2}\big| \  j< k,     \quad  | k^d-j^d -D | < nf(|D_r|)     \big\}\big| .
			\end{aligned}
		\end{equation*}
		To this we can now apply our Theorem 8 to continue from \eqref{p2k3} 
		\begin{equation*}\label{}
			\begin{aligned}
				\lesssim_{d,n}|A|^{2n-4-2d'}|D_r|^{\frac{2d'}{d}} f^{\frac{2}{d}}(|D_r|)\lesssim_{d,n} |A|^{2n-4-2d'}|D_r|^{\frac{2d'+1}{d}} \log^{\frac{1}{d}}|D_r| .
			\end{aligned}
		\end{equation*}
		Plugging this in  we can estimate 
		\begin{equation*}\label{p2k8}
			\begin{aligned}	
				{\bf S} \lesssim_{d,n} |A|^{2n-4-2d'} \sum_{\substack{j,k\in A\\ j<k}} \frac{ \log^{\frac{1}{d}}|k^d-j^d|}{ |k^d-j^d|^{\frac{1}{2}-\frac{2d'+1}{d}}}. 
			\end{aligned}
		\end{equation*}
		 In the last sum  the power of $|k^d-j^d|$ is $i/2d$ with $i=\{1,2,3,4\}$ depending on $d \pmod{4}$. In any case the summand in this last sum is a decreasing function of the type described in Lemma 8 on $[e^2,\infty)$.   So we can apply Lemma 8 twice to bound the last sum by
			\begin{equation*}\label{}
			\begin{aligned}	
			 \leq \sum_{\substack{10\leq  j<k\leq |A|+9}} \frac{ \log^{\frac{1}{d}}|k^d-j^d|}{ |k^d-j^d|^{\frac{1}{2}-\frac{2d'+1}{d}}}\lesssim_d \log^{\frac{1}{d}}|A| \sum_{\substack{1\leq  j<k\leq |A|}}  |k^d-j^d|^{\frac{2d'+1}{d}-\frac{1}{2}}. 
			\end{aligned}
		\end{equation*}
	To estimate the last sum 	we factorize, and apply a change of variables to obtain 
		\begin{equation*}\label{}
			\begin{aligned}	
				\leq 	\sum_{1<  k\leq |A|      }k^{(d-1)(\frac{2d'+1}{d}-\frac{1}{2})}\sum_{\substack{1\leq j<k  }  }\big(k-j\big)
				^{\frac{2d'+1}{d}-\frac{1}{2}}=	\sum_{1<  k\leq |A|      }k^{(d-1)(\frac{2d'+1}{d}-\frac{1}{2})}\sum_{\substack{1\leq j<k  }  }j
				^{\frac{2d'+1}{d}-\frac{1}{2}}.
			\end{aligned}
		\end{equation*}
	Using integrals we bound the inner sum and obtain 
		\begin{equation*}\label{}
			\begin{aligned}	
		   \lesssim	\sum_{1<  k\leq |A|      }k^{(d-1)(\frac{2d'+1}{d}-\frac{1}{2})}k
		^{\frac{2d'+1}{d}+\frac{1}{2}}= 	\sum_{1<  k\leq |A|      }k^{2d'+2-\frac{d}{2}}.
	\end{aligned}
\end{equation*}
		The exponent in this last sum is  $-1$ if $d\equiv 2 \pmod{4}$, and $> -1$ otherwise. So this sum is bounded by
		\begin{equation*}\label{}
			\begin{aligned}	
				\lesssim \begin{cases} |A|^{2d'+3-\frac{d}{2}}\log |A|  \quad  &\text{if} \quad  d\equiv 2 \pmod{4}, \\
					|A|^{2d'+3-\frac{d}{2}}  \quad  &\text{else}. 
				\end{cases}
			\end{aligned}
		\end{equation*}
	This finishes the estimation of {\bf S}, and gives exactly the desired bounds
		\begin{equation*}\label{}
			\begin{aligned}	
			{\bf S}	\lesssim_{d,n} \begin{cases} |A|^{2n-\frac{d}{2}-1}\log^{1+\frac{1}{d}} |A|  \quad  &\text{if} \quad  d\equiv 2 \pmod{4}, \\
				|A|^{2n-\frac{d}{2}-1}\log^{\frac{1}{d}} |A|  \quad  &\text{else}.
				\end{cases}
			\end{aligned}
		\end{equation*}
	Putting together all our estimates 	 finishes the proof.

	\end{proof}


	\section{The Arithmetic Component}

We  provide two proofs of Theorem 8. The first is by viewing it as a margin of sets of type \eqref{os}. The second is directly, without using any result from the  literature on lattice point problems. 

\begin{proof}\label{APROOF 1}
	
Plainly the cardinality of our set \eqref{os}  can be obtained from the equations
\begin{equation*}
	\begin{aligned}
		\big|\big\{ (j,k)\in \Z^2:  0<|k|^d-|j|^d\leq x    \big\}\big|&= 	\big|\big\{ (j,k)\in \Z^2:  0<|k|^d-|j|^d< x   \big \}\big|  \\   &+	\big|	\big\{ (j,k)\in \Z^2:  |k|^d-|j|^d= x    \big\}\big|
	\end{aligned}
\end{equation*} 
\begin{equation*}
	\begin{aligned}
			\big|\big\{ (j,k)\in \Z^2:  0<|k|^d-|j|^d< x   \big \}\big|&= 	4\big|\big\{ (j,k)\in \N^2:  0<k^d-j^d< x   \big \}\big|\\   &+	2\big|\big\{ k\in \N:  0<k^d< x  \big  \}\big|.
	\end{aligned}
\end{equation*} 
Plainly 
\begin{equation*}
	\big|\big\{ k\in \N:  0<k^d< x   \big\}\big|= x^{1/d}+\Oh(1).
\end{equation*} 
We also have 
\begin{equation*}
\big|\big\{ (j,k)\in \Z^2:  |k|^d-|j|^d= x    \big\}\big|=4	\big|\big\{ (j,k)\in \N^2:  k^d-j^d= x    \big\}\big|+\Oh(1).
\end{equation*} 
For the set on the right hand side to be nonempty, $x$ must be an integer. Then $k-j$ is a  divisor of $x$,  so it has at most $2d(x)$ different values. For a fixed value  $a=k-j$, we have $x=k^d-j^d=(j+a)^d-j^d,$ as $a\neq 0$, this means $j$ is the root of a degree $d-1$ polynomial. Once $j$ is fixed, so is $k$. Hence there are $2(d-1)d(x)$ different pairs $(j,k)$ in this set. But $2(d-1)d(x)\leq C_{d,\varepsilon}x^{\varepsilon}.$  Therefore
\begin{equation}\label{crv}
	\big	|\big\{ (j,k)\in \Z^2:  |k|^d-|j|^d= x    \big\}\big|=\Oh_{d,\varepsilon}(x^{\varepsilon}).
\end{equation} 
Using these we finally obtain
\begin{equation*}
	\begin{aligned}
	\big	|\big\{ (j,k)\in \N^2:  0<k^d-j^d< x   \big \}\big|	 =  \frac{1}{4}\big|\big\{ (j,k)\in \Z^2:  0<|k|^d-|j|^d\leq  x   \big \}\big|-\frac{1}{2}x^{1/d}+\Oh_{d,\varepsilon}(x^{\varepsilon}),
	\end{aligned}
\end{equation*} 
which by  \eqref{R-}, and  \eqref{nwkr}
\begin{equation}\label{eta}
	\begin{aligned}
		= \frac{1}{4}\Big[A_dx^{2/d}+B_d x^{1/(d-1)}+D_dF_d(x^{1/d})x^{1/d-1/d^2}\Big]-\frac{1}{2}x^{1/d}+\Oh_{d,\varepsilon}(x^{\frac{46}{73}\frac{1}{d}+\varepsilon}).
	\end{aligned}
\end{equation} 
Choosing $\varepsilon=1/100d$ in both \eqref{nwkr},\eqref{crv}, and considering $D_dF_d(x^{1/d})x^{1/d-1/d^2}$ as an error term  we obtain the estimate 
\begin{equation}\label{eta3}
	= \frac{1}{4}\Big[A_dx^{2/d}+B_d x^{1/(d-1)}\Big]-\frac{1}{2}x^{1/d}+\Oh_{d}(x^{1/d-1/d^2}).
\end{equation}

From this we will deduce a theorem on the set \eqref{ss2}. 
We split this estimate into two:
	\begin{equation}\label{p6te2}
		\begin{aligned}
			\sup_{D\leq E\leq 4D}	&\big|\big\{ (j,k) \in \N^2: j<k,  \   |k^d-j^d-E| <D \big\}\big| \\
			\sup_{4D\leq E\leq D^2}	&\big|\big\{ (j,k) \in \N^2: j<k,  \   |k^d-j^d-E| <D \big\}\big|	.
		\end{aligned}
	\end{equation}
	The first estimate is easier to do, and we start with this first estimate. It is bounded by
	\begin{equation}\label{a5d}
		\begin{aligned}
		\big|\big\{ (j,k) \in \N^2: \  0<k^d-j^d <5D \big\}\big|.
		\end{aligned}
	\end{equation}
	We apply \eqref{eta3} to immediately bound this by $C_dD^{2/d}.$ For the second estimate in \eqref{p6te2} we fix an $E$ in the described range, and then apply \eqref{eta3} for the values $E+D$ and $E-D$, and  subtract to obtain
	\begin{equation*}\label{}
		\big|\big\{ (j,k) \in \N^2: j<k,  \   E-D\leq k^d-j^d <E+D \big\}\big| 
	\end{equation*}
	\begin{equation}\label{fth}
		\begin{aligned}
			&\leq  \frac{1}{4}\Big[A_d\big[(E+D)^{2/d}-(E-D)^{2/d}\big] +B_d \big[(E+D)^{1/(d-1)}-(E-D)^{1/(d-1)}\big]\Big]\\ &-\frac{1}{2}\big[ (E+D)^{1/d}-(E-D)^{1/d} \big]+\Oh_{d}(E^{\frac{1}{d}-\frac{1}{d^2}}).
		\end{aligned}
	\end{equation}
	We have the estimates for $4D\leq E\leq D^2$  by mean value theorem 
	\begin{equation*}\label{}
		(E+D)^{\frac{2}{d}}-(E-D)^{\frac{2}{d}}\lesssim_d DE^{\frac{2}{d}-1}\leq D^{\frac{2}{d}}, \qquad  	(E+D)^{\frac{1}{d-1}}-(E-D)^{\frac{1}{d-1}}\lesssim_d DE^{\frac{1}{d-1}-1}\leq D^{\frac{1}{d-1}}.
	\end{equation*}                                                                                                                                           Plugging these into \eqref{fth} we obtain the bound $C_dD^{2/d}$.
	
\end{proof}

From the literature on lattice point problems, see for example \cite{gk}, we know that  for the error terms $\Delta_d$ bounds of order $\Oh(x^{\frac{1}{d}})$ are reachable without  recourse to    Fourier analysis, merely by carefully  decomposing the region, and  carefully   choosing the direction of the lines with which we  count the lattice points. We  show now that the same is true for  our Theorem 8. However the proof is considerably long and difficult, as our estimate concerns not just one fixed set but a whole class of sets. This forces a case by case analysis. Another issue is that for $d\geq 3$ the curves $y^d-x^d=m$ do not have the algebraic simplicity that they have when  $d=2.$ Once $d>2$ we no longer have easy rotation and inversion properties. All this comes through in the proof, most  obviously in our introduction and study of the function $g$.

	\begin{proof}
		
		We split this estimate into two as in \eqref{p6te2}, 
	and start with the   first. It is bounded by  \eqref{a5d}.
		We count the pairs of this set by fixing $j.$ 
		We observe that 
		\begin{equation*}\label{}
			\begin{aligned}
				(k-j)dj^{d-1}<  \big(k-j\big)\big(k^{d-1}+k^{d-2}j+\ldots +kj^{d-2}+j^{d-1}\big)=k^d-j^d <5D.
			\end{aligned}
		\end{equation*}
		Therefore $j<5D^{\frac{1}{d-1}}$, and  for each such $j$ 
		\begin{equation*}\label{}
			\begin{aligned}
				1\leq	k-j<5Dj^{1-d}.
			\end{aligned}
		\end{equation*}
		Hence for a fixed $j$ there are at most $5Dj^{1-d}$ pairs $(j,k).$ We will use this estimate for $j>D^{1/d}.$  For $j\leq D^{1/d}$ we observe that 
		\begin{equation*}\label{}
			\begin{aligned}
				k^d-j^d <5D\implies k<(j^d+5D)^{1/d}\implies k<6D^{1/d}. 
			\end{aligned}
		\end{equation*}
		Thus for these $j$ there are  at most $6D^{1/d}$ pairs $(j,k).$ Now summing all these we can
		bound our set by
		\begin{equation*}\label{}
			\begin{aligned}
				\sum_{1\leq j < 5D^{\frac{1}{d-1}}}\big|\big\{ k\in \N:   \   0<k^d-j^d <5D \big\}\big| 
				\leq  \sum_{1\leq j\leq D^{1/d}}  6D^{1/d} + 5D\cdot \sum_{D^{\frac{1}{d}}<j < 5D^{\frac{1}{d-1}}} j^{1-d}.
			\end{aligned}
		\end{equation*}
		The first sum is plainly bounded by $6D^{2/d}$.  The second sum can be estimated by  
		\begin{equation*}\label{}
			\begin{aligned}
			D^{-1+\frac{1}{d}}	+\int_{D^{1/d}}^{\infty} x^{1-d}dx\leq 2D^{-1+\frac{2}{d}}.
			\end{aligned}
		\end{equation*}
		So we obtain the final bound $16D^{2/d}.$

		Now we turn to the second estimate in \eqref{p6te2}.
		We will count this set via differences $b=k-j$ 	as in the proof of Theorem 4. Since $b^d=(k-j)^d<k^d-j^d<E+D$ we must have
		$1\leq b<(E+D)^{1/d}.$ Then  for a fixed $E\in [4D,D^2],$ the second set in \eqref{p6te2} is bounded by 
		\begin{equation*}\label{}
			\begin{aligned}
			\leq \sum_{1\leq b<(E-2D)^{\frac{1}{d}}}+\sum_{(E-2D)^{\frac{1}{d}}\leq b<(E+D)^{\frac{1}{d}}}\big|\big\{ j \in \N:    |(j+b)^d-j^d-E| <D \big\}\big| 
			={\bf I+II}.
			\end{aligned}
		\end{equation*}

		To deal with these sets we introduce the function 
		\begin{equation*}\label{}
			\begin{aligned}
				g(x):=(x+b)^d-x^d-b^d=\sum_{j=1}^{d-1}{d \choose j}x^{d-j}b^j,    \quad x\geq 0.
			\end{aligned}
		\end{equation*}
		This function is strictly increasing for $x\geq 0$ with $g(0)=0$, and indeed is a bijection on $[0,\infty)$.  So  it has a strictly increasing inverse  $g^{-1}:[0,\infty)\rightarrow [0,\infty) $. Observe that
		\begin{equation}\label{comgg}
			\begin{aligned}
				g(x)	\leq xg'(x)=\sum_{j=1}^{d-1}{d \choose j}(d-j)x^{d-j}b^j\leq dg(x),    \quad x\geq 0.
			\end{aligned}
		\end{equation}
		Also   for $x\geq 0$,
		\begin{equation}\label{p6l1}
			\begin{aligned}
				g(x)\geq dx^{d-1}b+dxb^{d-1},  \  \text{and} \ \  g'(x)\geq db^{d-1}.
			\end{aligned}
		\end{equation}
		Furthermore from the first of these we obtain 
		\begin{equation}\label{p6l4}
			\begin{aligned}
				g(x)\geq dbx^{d-1} \implies    x\geq g^{-1}(dbx^{d-1})\implies    \Big(\frac{x}{db}\Big)^{\frac{1}{d-1}}\geq g^{-1}(x),\\ \
				g(x)\geq dxb^{d-1} \implies    x\geq g^{-1}(dxb^{d-1})\implies    \frac{x}{db^{d-1}}\geq g^{-1}(x).
			\end{aligned}
		\end{equation}
		
		We now  apply these observations on  $g$ to obtain  our results. We first deal with {\bf II}.  Fix  $b$. 	
		\begin{equation}\label{p6l2}
			\begin{aligned}
				\big|\big\{ j \in \N:    |(j+b)^d-j^d-E| <D \big\} \big|  =	&\big|\big\{ j \in \N:    E-D-b^d<g(j) <E+D-b^d \big\} \big| \\  \leq	&\big|\big\{ j \in \N:   0<g(j) <E+D-b^d \big\} \big|.
			\end{aligned}
		\end{equation}
		By \eqref{p6l1}  
		\begin{equation*}\label{}
			\begin{aligned}
				g(j) <E+D-b^d  \implies	j<\frac{E+D-b^d}{db^{d-1}}.
			\end{aligned}
		\end{equation*}
		Since $j>0$, and $b^d\geq E-2D\geq 2D$  we bound the last term of \eqref{p6l2} by 
		\begin{equation*}\label{}
			\begin{aligned}
				\big|\big\{ j\in \N:   1\leq j < \frac{E+D-b^d}{db^{d-1}} \big\} \big|  \leq \frac{E+D-b^d}{db^{d-1}} \leq \frac{3D}{db^{d-1}}  \leq \frac{3D^{1/d}}{d}\leq D^{1/d}.
			\end{aligned}
		\end{equation*}
		Observe that the cardinality of integers $b$ with our assumption can be bounded from
		\begin{equation*}\label{}
			\begin{aligned}
				3D=(E+D)-(E-2D)=	\big[ (E+D)^{\frac{1}{d}}-(E-2D)^{\frac{1}{d}}  \big] \Big[\sum_{j=1}^{d}  (E+D)^{\frac{d-j}{d}}(E-2D)^{\frac{j-1}{d}}\Big],
			\end{aligned}
		\end{equation*}
		implying that 
		\begin{equation*}\label{}
			\begin{aligned}
				(E+D)^{\frac{1}{d}}-(E-2D)^{\frac{1}{d}}   < 3D(E+D)^{\frac{1}{d}-1}<3D^{\frac{1}{d}}.
			\end{aligned}
		\end{equation*}
		So we bound 
		\begin{equation*}\label{}
			\begin{aligned}
				{\bf II}\leq (3D^{\frac{1}{d}}+1)D^{1/d}\leq 4D^{2/d}.
			\end{aligned}
		\end{equation*}
		
		We turn to ${\bf I}.$ Again we fix $b$.
		\begin{equation*}\label{}
			\begin{aligned}
				\big|\big\{ j \in \N:    |(j+b)^d-j^d-E| <D \big\} \big|   &=	\big|\big\{ j \in \N:    E-D-b^d<g(j) <E+D-b^d \big\} \big| \\  &=	\big|\big\{ j \in \N:  g^{-1}(E-D-b^d) <j <g^{-1}(E+D-b^d) \big\} \big|.
			\end{aligned}
		\end{equation*}
		We measure the cardinality of this set by the  mean value theorem
		\begin{equation}\label{p6l3}
			\begin{aligned}
				|g^{-1}(E+D-b^d)- g^{-1}(E-D-b^d)| = 2D (g^{-1})'(x) ,
			\end{aligned}
		\end{equation}
		for some $E-D-b^d< x< E+D-b^d.$  Observe that since both $g',g^{-1}$ are strictly increasing on the positive real axis, $	(g^{-1})'$ is strictly decreasing there. So \eqref{p6l3} can be bounded by $2D (g^{-1})'(E-D-b^d).$ 
		We have from \eqref{comgg} for $y>0$
		\begin{equation*}\label{}
			\begin{aligned}
				(g^{-1})'(y)=\frac{1}{g'(g^{-1}(y))}\leq \frac{g^{-1}(y)}{g(g^{-1}(y))}=\frac{g^{-1}(y)}{y}.
			\end{aligned}
		\end{equation*}
		Combining this with \eqref{p6l4}
		\begin{equation}\label{p6l5}
			\begin{aligned}
				(g^{-1})'(E-D-b^d) \leq 	\frac{g^{-1}(E-D-b^d)}{E-D-b^d}\leq \min \Big\{\frac{1}{(E-D-b^d)^{\frac{d-2}{d-1}}(db)^{\frac{1}{d-1}}}, \frac{1}{db^{d-1}}  \Big\} .
			\end{aligned}
		\end{equation}
		We can bound {\bf I}  as follows
		\begin{equation*}\label{}
			\begin{aligned}
				{\bf I}\leq \sum_{1\leq b<(E-2D)^{\frac{1}{d}}}1+	{2D}(g^{-1})'(E-D-b^d)\leq (E-2D)^{\frac{1}{d}} +  2D\Big[\sum_{1\leq b<(E-2D)^{\frac{1}{d}}}\frac{g^{-1}(E-D-b^d)}{(E-D-b^d)}\Big].
			\end{aligned}
		\end{equation*}
		As we assumed $E\leq D^2$ we can bound $(E-2D)^{1/d}\leq D^{2/d}$. We decompose the last  sum in paranthesis, and bound it by \eqref{p6l5}
		\begin{equation*}\label{}
			\begin{aligned}
				&\leq \sum_{1\leq b<D^{\frac{1}{d}}}\frac{1}{(E-D-b^d)^{\frac{d-2}{d-1}}b^{\frac{1}{d-1}}}+	\sum_{D^{\frac{1}{d}}\leq b<(E-2D)^{\frac{1}{d}}}\frac{1}{db^{d-1}}\\ &\leq  \frac{2}{(E-D)^{\frac{d-2}{d-1}}}\Big[\sum_{1\leq b<D^{\frac{1}{d}}}b^{-\frac{1}{d-1}}\Big]+	\frac{1}{d}\Big[\sum_{D^{\frac{1}{d}}\leq b<(E-2D)^{\frac{1}{d}}}b^{1-d}\Big].
			\end{aligned}
		\end{equation*}
			We can estimate  these sums by 
		\begin{equation*}\label{}
			\begin{aligned}
				\sum_{1\leq b<D^{\frac{1}{d}}}b^{-\frac{1}{d-1}}\leq 1+ \int_{1}^{D^{\frac{1}{d}}} b^{-\frac{1}{d-1}}db=1+\frac{d-1}{d-2}\Big[b^{\frac{d-2}{d-1}} \Big|^{D^{\frac{1}{d}}}_1 \Big]\leq   2D^{\frac{d-2}{d(d-1)}},
			\end{aligned}
		\end{equation*}
		and
		\begin{equation*}\label{}
			\begin{aligned}
				\sum_{D^{\frac{1}{d}}\leq b< 
					(E-2D)^{\frac{1}{d}}}b^{1-d}\leq D^{\frac{1-d}{d}}+ \int_{D^{\frac{1}{d}}}^{\infty}b^{1-d}db \leq 2D^{-1+\frac{2}{d}}.
			\end{aligned}
		\end{equation*}
		Combining all of these we can write
		\begin{equation*}\label{}
			\begin{aligned}
				{\bf I} \leq D^{\frac{2}{d}} +2D\Big[\frac{4D^{\frac{d-2}{d(d-1)}}}{(E-D)^{\frac{d-2}{d-1}}}+	\frac{2}{d}D^{-1+\frac{2}{d}}\Big]\leq D^{\frac{2}{d}} +8D^{1+\frac{d-2}{d(d-1)}-\frac{d-2}{d-1}}+	2D^{\frac{2}{d}}\leq 11D^{\frac{2}{d}}.
			\end{aligned}
		\end{equation*}
		This finishes the proof.

	\end{proof}

We  turn to proving our last theorem. The proof follows the same lines along our first proof of Theorem 8.

\begin{proof}

Let $D^s\geq 2D.$	We	apply \eqref{eta3} for the values $D^s+D$ and $D^s-D$, and  subtract to bound the left hand side of \eqref{p6sl3} by 
	\begin{equation}\label{fths2}
		\begin{aligned}
			&\leq  \frac{1}{4}\Big[A_d\big[(D^s+D)^{2/d}-(D^s-D)^{2/d}\big] +B_d \big[(D^s+D)^{1/(d-1)}-(D^s-D)^{1/(d-1)}\big]\Big]\\ &-\frac{1}{2}\big[ (D^s+D)^{1/d}-(D^s-D)^{1/d} \big]+\Oh_{d}(D^{\frac{s}{d}[1-\frac{1}{d}]}).
		\end{aligned}
	\end{equation}
	We have the estimates   by the mean value theorem 
	\begin{equation*}\label{}
		(D^s+D)^{\frac{2}{d}}-(D^s-D)^{\frac{2}{d}}\leq 4 D^{1+s(\frac{2}{d}-1)}, \qquad  (D^s+D)^{\frac{1}{d-1}}-(D^s-D)^{\frac{1}{d-1}}\leq 4 D^{1+s(\frac{1}{d-1}-1)}.
	\end{equation*}
	Of these  plainly the first one dominates. Comparing it with the error term in \eqref{fths2} yields the desired result.
	
		If $D^s< 2D$ then \eqref{eta3} can be applied with $x=3D$ to show that our set has cardinality at most $C_dD^{2/d}$. But  $D^s< 2D$ leads to 
	\[D^{\frac{2}{d}}=DD^{\frac{2}{d}-1}\leq D(D^s/2)^{\frac{2}{d}-1}\leq 2D^{1+s(\frac{2}{d}-1)}.\]
	This finishes the proof.

\end{proof}

\end{document}